\documentclass[11pt,leqno]{amsart}

\usepackage{amssymb}
\usepackage{amsmath}
\usepackage{amsthm}
\usepackage{latexsym}
\usepackage{amscd, xypic}
\usepackage[dvipsnames]{xcolor}
\usepackage{tikz-cd} %para diagramas conmutativos \begin{tikzcd}[]
\tikzcdset{arrow style=tikz}

\begin{document}

\newtheorem{thm}{Theorem}
\newtheorem{prop}[thm]{Proposition}
\newtheorem{lem}[thm]{Lemma}
\newtheorem{cor}[thm]{Corollary}
\theoremstyle{definition}
\newtheorem{defi}[thm]{Definition}
\newtheorem{ex}[thm]{Example}
\newtheorem{rem}[thm]{Remark}
\newtheorem{que}[thm]{Question}
\newtheorem*{ass}{Assumption}

\newcommand{\HOM}{\mbox{HOM}}
\newcommand{\Hom}{\mbox{Hom}}
\newcommand{\END}{\mbox{END}}
\newcommand{\Ker}{\mbox{Ker}}
\newcommand{\Image}{\mbox{Im}}
\newcommand{\iso}{\mbox{Iso}}
\newcommand{\m}{^{-1}}
\newcommand{\G}{\mathcal {G}}
\newcommand{\U}{\mathcal {U}}
\newcommand{\N}{\mathbb {N}}
\newcommand{\si}{\sigma}
\newcommand{\rh}{\rho}
\newcommand{\ta}{\tau}
\newcommand{\id}{\mbox{id}}
\newcommand{\af}{\alpha}

\newcommand{\GR}{\mbox{$\G$-$R$}}
\newcommand{\AbG}{\mbox{$\mbox{Ab}_{\G}$}}
\newcommand{\Ab}{\mbox{Ab}}
\newcommand{\rmd}{\mbox{$R$-md}}
\newcommand{\mdr}{\mbox{md-$R$}}
\newcommand{\rmds}{\mbox{$R$-md-$S$}}
\newcommand{\rmod}{\mbox{$R$-mod}}
\newcommand{\smod}{\mbox{$S$-mod}}
\newcommand{\rmods}{\mbox{$R$-mod-$S$}}
\newcommand{\modr}{\mbox{mod-$R$}}
\newcommand{\grmod}{\mbox{$\G$-$R$-mod}}
\newcommand{\gsmod}{\mbox{$\G$-$S$-mod}}
\newcommand{\gmodr}{\mbox{$\G$-mod-$R$}}
\newcommand{\grmods}{\mbox{$\G$-$R$-mod-$S$}}
\newcommand{\grmodr}{\mbox{$\G$-$R$-mod-$R$}}
\newcommand{\gmods}{\mbox{$\G$-mod-$S$}}
\newcommand{\grmd}{\mbox{$\G$-$R$-md}}
\newcommand{\gsmd}{\mbox{$\G$-$S$-md}}
\newcommand{\gmdr}{\mbox{$\G$-md-$R$}}
\newcommand{\gmds}{\mbox{$\G$-md-$S$}}
\newcommand{\grmds}{\mbox{$\G$-$R$-md-$S$}}

\newcommand{\rumod}{R\text{-\textup{\textbf{umod}}}}
\newcommand{\sumod}{S\text{-\textup{\textbf{umod}}}}
\newcommand{\modgr}{\mbox{{\bf mod}-}\G\mbox{-$R$}}
\newcommand{\grumod}{\G\textup{-}\rumod}
\newcommand{\gsumod}{\G\textup{-}\sumod}
\newcommand{\umodgr}{\G\text{-\textup{\textbf{umod}}}\textup{-}R}
\newcommand{\umodgs}{\G\text{-\textup{\textbf{umod}}}\textup{-}S}
\newtheorem{assertion}[thm]{{\sc Assertion}}
\newcommand\restr[2]{{% we make the whole thing an ordinary symbol
  \left.\kern-\nulldelimiterspace % automatically resize the bar with \right
  #1 % the function
  \vphantom{\big|} % pretend it's a little taller at normal size
  \right|_{#2} % this is the delimiter
  }}

\title[Graded modules over object-unital groupoid graded rings]{Graded modules over object-unital groupoid graded rings}

\author[J. Cala]{Juan  Cala }
\address{Escuela de Matem\'{a}ticas,
Universidad Industrial de Santander,
Carrera 27 Calle 9,
Edificio Camilo Torres
Apartado de correos 678,
Bucaramanga, Colombia}
\email{{\scriptsize jccalab@gmail.com}}

\author[P. Lundstr\"{o}m]{Patrik Lundstr\"{o}m}
\address{University West,
Department of Engineering Science, 
SE-46186 Trollh\"{a}ttan, Sweden}
\email{{\scriptsize patrik.lundstrom@hv.se}}

\author[H. Pinedo]{H. Pinedo}
\address{Escuela de Matem\'{a}ticas,
Universidad Industrial de Santander,
Carrera 27 Calle 9,
Edificio Camilo Torres
Apartado de correos 678,
Bucaramanga, Colombia}
\email{{\scriptsize hpinedot@uis.edu.co}}



\begin{abstract}
In a previous article (see \cite{CNP}), 
we introduced and analyzed ring-theoretic properties of
object unital $\G$-graded rings $R$, where $\G$ is a groupoid.
In the present article, we analyze the category $\grmod$ of unitary $\G$-graded modules over such rings.
Following ideas developed earlier
by one of the authors in \cite{lundstrom2004}, 
we analyse the forgetful functor $U \colon \grmod \to \rmod$ and aim to determine properties 
$\mathcal{P}$ for which the following implications are valid for modules $M$ in $\grmod$: 
$M$ is $\mathcal{P}$ $\Rightarrow$ $U(M)$ is $\mathcal{P}$; 
$U(M)$ is $\mathcal{P}$ $\Rightarrow$ $M$ is $\mathcal{P}$. 
Here we treat the cases when $\mathcal{P}$ is any of the properties: 
direct summand, projective, injective, free, simple and semisimple. 
Moreover, graded versions of results concerning classical module theory are established, 
as well as some structural properties related to the category $\grmod$.
\end{abstract}

\maketitle

\section{Introduction}

Let $R$ denote an associative, but not necessarily unital, ring, 
and let $M$ be a left $R$-module.
If $R$ is unital, then we let $1_R$ denote the multiplicative identity of $R$.
If $X \subseteq R$ and $Y \subseteq M$, then we let $XY$ denote
the set of finite sums of elements of the form $xy$ 
for $x \in X$ and $y \in Y$. 
Following \cite{AnM}, we say that $M$ is {\it unitary} if $RM=M$.
We denote by $\rmd$ ($\rmod$) the category having (unitary) left 
$R$-modules as objects and $R$-module homomorphisms as morphisms.
Analogously, the category $\mdr$ ($\modr$) of right (unitary)
modules is defined.

Let $G$ be a group. 
Recall that $R$ is said to be {\it graded by $G$} (or $G$-\emph{graded})
if for all $g \in G$ there is an additive subgroup $R_g$ of $R$ such that $R = \bigoplus_{g \in G} R_g$
and for all $g,h \in G$ the inclusion $R_g R_h \subseteq R_{gh}$ holds.
The class of group graded rings contains numerous important mathematical structures, such as
polynomial rings, skew and twisted group rings, crossed products and partial versions of these
(see e.g. \cite{batista2017,N82,nastasescu2004} and the references therein).
Therefore, a theory of group graded rings can be applied to the study of completely different types of constructions. 
This not only gives new results for all of these constructions
simultaneously, but also serves as a unification of a multitude of known theorems concerning these.

Many relevant examples of rings, for instance rings of matrices, crossed
product algebras defined by separable extensions, crossed groupoid algebras,
including twisted and skew groupoid algebras, and partial versions of these, 
are not, in any natural way, graded by groups, but
instead by groupoids (see for instance \cite{lundstrom2004}, \cite{lundstrom2005}, \cite{lundstrom2006} and \cite{NyOP22018}).
Also, many of these structures are non-unital.
This inspired us in \cite{CNP} to introduce and analyse the class of {\it object unital
groupoid graded rings}, thereby extending some of the results from 
\cite{lundstrom2004,lundstrom2005,lundstrom2006} to cover non-unital rings.
Let us briefly describe these structures.

A groupoid $\G$ is a small category with the property that all morphisms are isomorphisms. 
Equivalently, this can be defined by saying that $\G$ is a set equipped with
a unary operation $\G \ni \sigma \mapsto \sigma^{-1} \in \G$ (inversion) and a partially defined multiplication
$\G \times \G \ni (\sigma,\tau) \mapsto \sigma \tau \in \G$ (composition) such that
for all $\si,\ta,\rh \in \G$ the following four axioms hold:
(i) $(\si^{-1})^{-1} = \si$;
(ii) if $\si \tau$ and $\ta \rh$ are defined, then $(\si \ta) \rh$ and $\si (\ta \rh)$ are defined and $(\si \ta) \rh=\si (\ta \rh)$;
(iii) the {\it domain} $d(\si) := \si^{-1} \si$ is always defined and if $\si \tau$ is defined, then $d(\si) \tau = \tau$;
(iv) the {\it range} $r(\ta) := \ta \ta^{-1}$ is always defined and if $\si \tau$ is defined, then $\si r(\ta) = \si$.
The maps $d$ and $r$ have a common image denoted by $\G_0$, which is called the \textit{unit space} of $\G$.
The set $\G_2 = \{ (\si,\ta) \in \G \times \G \mid \mbox{$\sigma \tau$ is defined} \}$ 
is called the set of {\it composable pairs} of $\G$. 
For more details about groupoids, the interested reader may consult for example \cite{L} or \cite{ren}.  

Recall from \cite{lundstrom2004}  that
a ring $R$ is said to be \emph{graded by $\G$} (or \emph{$\G$-graded}) if there 
for all $\si \in \G$ is an additive subgroup $R_{\si}$ of $R$ such that $R = \bigoplus_{\si \in \G} R_{\si}$
and for all $\si,\ta \in \G$ the inclusion $R_{\si} R_{\ta} \subseteq R_{\si \ta}$ holds,
if $(\si,\ta) \in \G_2$, and $R_{\si} R_{\ta} = \{ 0 \}$, otherwise.
From \cite{CNP} we say  that a $\G$-graded ring $R$ is  {\it object unital} if 
for all $e \in \G_0$ the ring $R_e$ is unital and for all $\si \in \G$ and all 
$r \in R_{\si}$ the equalities $1_{R_{r(\si)}} r = r 1_{R_{d(\si)}} = r$ hold.

Suppose that $R$ is a $\G$-graded ring and $M$ is a left $R$-module.
Recall from \cite{lundstrom2004} that the module $M$ is said to be 
{\it graded by $\G$} (or $\G$-graded) if there for all $\si \in \G$ is
an additive subgroup $M_\si$ of $M$ such that $M = \bigoplus_{\si \in \G} M_\si$
and for all $\si,\tau \in \G$ the inclusion $R_{\si} M_{\tau} \subseteq M_{\si \ta}$ holds,
if $(\si,\ta) \in \G_2$, and $R_{\si} M_{\ta} = \{ 0 \}$, otherwise.
If $N$ is another left $R$-module graded by $\G$, then a left $R$-module homomorphism
$f : M \to N$ is said to be \emph{graded} if for all $\si \in \G$ the inclusion $f(M_{\si}) \subseteq N_{\si}$ holds.
The collection of (unitary) $\G$-graded left $R$-modules and the collection of graded homomorphisms
together form an abelian category which we denote by 
$\grmd$ ($\grmod$). 
In fact, it is not hard to show that $\grmod$ is even a Grothendieck category.
Analogously, the category $\gmdr$ ($\gmodr$) of (unitary) 
$\G$-graded right $R$-modules is defined.

A natural class of functors to study from a categorical perspective are the forgetful functors,
which simply forget parts of the structure.
In this article, we study the {\it ungrading} functor
$U \colon \grmod \to \rmod$
which is defined by forgetting the grading. 
More precisely, we wish to answer the following 
\begin{que}\label{thequestion}
Suppose that $\G$ is a groupoid and $R$ is an object unital $\G$-graded ring.
For which graded modules $M$ in $\grmod$ %(or $\gmodr$) 
and for what properties $\mathcal{P}$ are either of the 
following two implications valid? 
\begin{equation*}
\mbox{$M$ has property $\mathcal{P}$ $\Longrightarrow$ $U(M)$ has property $\mathcal{P}$}
\end{equation*}
\begin{equation*}
\mbox{$U(M)$ has property $\mathcal{P}$ $\Longrightarrow$ $M$ has property $\mathcal{P}$}
\end{equation*}
\end{que}
When $R$ is unital and $\G$ is a group, 
then Question \ref{thequestion} has been investigated for many different properties $\mathcal{P}$ including:
direct summand, free, finitely generated, finitely presented, projective, injective,
essential, small and flat (see \cite[Sections I.2--I.3]{N82}).
Several of these results have been extended to the case 
when $\G$ is a groupoid and $R$ is unital (see \cite{lundstrom2004,lundstrom2005,lundstrom2006}).
The aim of this article is to establish these and other results in the non-unital situation. 
Here is a detailed outline of the article.

In Section \ref{sec:preliminaries}, 
we state our conventions on rings and modules
and  fix some notation with respect to these objects.
After that, we define various concepts related to graded
rings and modules. We also give some examples of graded rings
(see Example \ref{firstexample} and Example \ref{secondexample}).
At the end of the section, we prove a result
(see Proposition \ref{unitarymodule}) that connects
the categories $\grmd$ and $\grmod$.

In Section \ref{sec:homomorphisms}, we
introduce suspension for modules in $\grmd$
and we prove some auxiliary results about these
(see Propositions \ref{prop:suspensionunitary}-\ref{prop:Ggroup}). 
After that, we introduce the additive group of semi-graded maps
between objects in $\grmd$ 
(see Definition \ref{def:semigraded})
and we prove some 
structural results for these objects
(see Prop. \ref{prophom},
Prop. \ref{prop:finite},
Prop. \ref{propactions},
Prop. \ref{propalpha}
and Cor. \ref{coralpha}).
At the end of the section, we state and prove a groupoid graded 
version of a well-known result regarding the 
hom and tensor functors for bimodules
(see Proposition \ref{adjunction}).

In Section \ref{sec:theungradingfunctor},
we deal with Question \ref{thequestion},
in a non-unital situation, that is we
analyse properties of the ungrading functor
$U$ from $\grmd$ (or $\grmod$) to $\rmd$ (or $\rmod$),
including the properties
direct summand, free, finitely generated, finitely presented, projective, injective,
essential, small and flat.

In Section \ref{sec:semisimplicity},
we prove a graded non-unital version of a classical result 
concerning semisimple unital rings 
(see Proposition \ref{ss2}). 
After that, we show that semisimplicity is  well behaved
under the preimage of the forgetful functor 
(see Proposition \ref{preimageU}).
At the end of this section, we obtain a result
relating graded simplicity of $R$ (as an object in $\grmod$)
to graded injectivity and graded projectivity 
(see Proposition \ref{equivalent}).

\section{Preliminaries}\label{sec:preliminaries}

%In this section, we state our conventions on rings and modules
%and we fix some notation concerning these objects.
%After that, we define various concepts related to graded
%rings and modules. We also give some examples of graded rings
%(see Example \ref{firstexample} and Example \ref{secondexample}).
%At the end of the section, we prove a result
%(see Proposition \ref{unitarymodule}) that connects
%the categories $\grmd$ and $\grmod$.  YA SE DIJO LO QUE SE IVA A HACER EN ESTA SECCIÓN

\subsection{Rings and modules}
Let $R$ be a ring.
By this we mean that $R$ is associative but not necessarily unital.
Then $R$ is called {\it idem\-potent} if $R R = R$.
Following Fuller \cite{Fu}, we say that $R$ has {\it enough idempotents} if there exists a set 
$\{ e_i \}_{i \in I}$ of orthogonal idempotents in $R$ (called a complete set of idempotents for $R$) 
such that $R = \bigoplus_{i\in I} R e_i = \bigoplus_{i \in I} e_i R.$
Following \'{A}nh and M\'{a}rki \cite{AnM}, we say that $R$ is {\it locally unital} if for all $n \in \N$ and all
$r_1,\ldots , r_n \in R$ there is an idempotent $e \in R$ such that for all $i \in \{1, \ldots , n\}$
the equalities $e r_i = r_i e = r_i$ hold.
Following Tominaga \cite{tominaga1976}, we say that $R$ is {\it s-unital} if for all $r \in R$ the relation $r \in Rr \cap rR$ holds.
The following chain of implications hold (see e.g. \cite{Ny2019}) for all rings
\begin{equation*}\label{implications}
{\small \mbox{unital} \Rightarrow \mbox{enough idempotents} \Rightarrow \mbox{locally unital} 
\Rightarrow \mbox{$s$-unital} \Rightarrow \mbox{idempotent.}}
\end{equation*}
Let $M$ denote a left $R$-module. By this we mean that $M$ is an additive group
equipped with a biadditive map $R \times M \ni (r,m) \mapsto rm \in M$.
In that case, we write ${}_R M$ to indicate this.
We say that ${}_R M$ is unitary if $RM = M$. 
Analogously, (unitary) right $R$-modules $M_R$ are defined.

If $M$ and $N$ are left (or right) $R$-modules, then we 
left $\Hom_R(M,N)$ denote the additive group of $R$-linear maps $M \to N$.
We denote by $\rmd$ ($\rmod$) the category having (unitary) left 
$R$-modules as objects and $R$-module homomorphisms as morphisms.
Analogously, the category $\mdr$ ($\modr$) of right (unitary)
modules is defined.
Let $S$ be another ring.
We say that an additive group $M$ is an (unitary) $R$-$S$-bimodule
if $M$ is simultaneously a (unitary) left $R$-module and a (unitary) 
right $S$-module which also is balanced, that is having the 
property that $(r m)s = r(m s)$ for $r \in R$, $m \in M$ and $s \in S$. 
We denote by $\rmds$ ($\rmods$) the category having (unitary) 
$R$-$S$-bimodules as objects and $R$-$S$-bimodule 
homomorphisms as morphisms.

Since we are considering modules over non-unital rings,
freeness in the usual sense, that is, that the module has a basis, seems too restrictive.
Instead, we say that a left $R$-module $M$ is {\it weakly free (of finite type)} if $M$ is isomorphic,
as a module, to a (finite) direct sum of copies of $R$.

\subsection{Graded rings}

Let $R = \bigoplus_{\si \in \G} R_{\si}$ be a ring graded by the groupoid $\G$. %Recall from the introduction that this means that there for all $\si \in \G$ is an additive subgroup $R_{\si}$ of $R$ such that $R = \bigoplus_{\si \in \G} R_{\si}$ and for all $\si,\ta \in \G$ the inclusion $R_{\si} R_{\ta} \subseteq R_{\si \ta}$ holds, if $(\si,\ta) \in G_2$, and $R_{\si} R_{\ta} = \{ 0 \}$, otherwise.
Then $R$ is said to be {\it object unital} if 
for all $e \in \G_0$ the ring $R_e$ is unital and for all $\si \in \G$ and all 
$r \in R_{\si}$ the equalities $1_{R_{r(\si)}} r = r 1_{R_{d(\si)}} = r$ hold.
The set $H(R)=\bigcup_{\si \in \G} R_\si$ is called the set of \emph{homogeneous elements of $R$}. 
If $r \in R_\si \setminus \{ 0 \}$, then we say that $r$ is of \emph {degree $\si$} and write $\deg(r)=\si$. 
Any $r \in R$ has a unique decomposition $r = \sum_{\si \in \G} r_{\si}$, where
$r_{\si} \in R_{\si}$, for $\si \in \G$, and all but finitely many of the $r_{\si}$ are zero.  

A non-empty subset $\mathcal{H}$ of $\G$ is called a {\it subgroupoid} of $\G$ if
(i) $h \in \mathcal{H}$ $\Rightarrow$ $h^{-1} \in \mathcal{H}$, and
(ii) $h_1,h_2 \in \mathcal{H} \cap \G_2$ $\Rightarrow$ $h_1 h_2 \in \mathcal{H}$.
In that case, $\mathcal{H}$ is called a {\it wide} subgroupoid of $\G$ if $\mathcal{H}_0 = \G_0$.

\begin{prop}[{\cite[Proposition 4]{CNP}}]\label{objectunitalnonzero}
If $R$ is object unital and we put 
\[
\mathcal{H} = \{ \si \in G \mid 1_{R_{r(\si)}} \neq 0 \ \mbox{and} \ 1_{R_{d(\si)}} \neq 0 \},
\] 
then $\mathcal{H}$ is a subgroupoid of $\G$ and $R = \bigoplus_{\si \in \mathcal{H}} R_{\si}$.
The subgroupoid $\mathcal{H}$ is wide if and only if for all $e \in \G_0$ the element $1_e$ is non-zero. 
\end{prop}

In light of Proposition \ref{objectunitalnonzero}, we will from now on make the following 
\begin{ass} 
If $R$ is object unital, then for all $e \in \G_0$, $1_{R_e} \neq 0$.
\end{ass}

There are many examples of object unital groupoid graded rings (see \cite{CNP}).
Here, we describe two large such classes.

\begin{ex}\label{firstexample}
Suppose that we are given a collection $A = ( A_e )_{e \in \G_0}$ of non-zero unital rings $A_e$
and put $A_0 = \bigoplus_{e \in \G_0} A_e$.
For all $e,f \in \G_0$ let $\iso_{e,f}(A)$ denote the set of ring isomorphisms $A_f \to A_e$ (respecting identity elements).
We let $\iso(A)$ denote the disjoint union $\biguplus_{e,f \in \G_0} \iso_{e,f}(A)$ and
we define a groupoid structure on $\iso(A)$ in the following way.
Take $e,e',f,f' \in \G_0$. The partial composition on $\iso(A)$
is defined to be the usual composition of functions
$\iso_{e,f}(A) \times \iso_{e',f'}(A) \to \iso_{e,f'}(A)$, when $f = e'$, and otherwise undefined.
By an {\it object crossed system} we mean a quadruple $( A , \G , \alpha , \beta )$ where
$\alpha : \G \to \iso(A)$ and $\beta : \G_2 \to U^{\rm gr}(A_0)$ are maps satisfying the 
following axioms for all $(\si,\ta,\rho) \in \G_3$ and all $a \in A_{d(\ta)}$
\begin{itemize}

\item $\alpha_{\si} : A_{d(\si)} \to A_{r(\si)}$ and $\alpha_e = {\rm id}_{A_e}$ for $e \in \G_0$;

\item $\beta_{\si , \ta} \in U( A_{r(\si)} )$ and $\beta_{\si , d(\si)} = \beta_{ r(\si) , \si } = 1_{A_{r(\si)}}$;

\item $\alpha_{\si} ( \alpha_{\ta} (a) ) = \beta_{\si,\ta} \alpha_{\si \ta} (a) \beta_{\si,\ta}^{-1}$;

\item $\beta_{\si,\ta} \beta_{\si\ta,\rho} = \alpha_{\si}( \beta_{\ta,\rho} ) \beta_{\si,\ta\rho}$.

\end{itemize}
The map $\alpha$ is called a {\it weak action} of $\G$ on $A$ and $\beta$ is called an $\alpha$-{\it cocycle}.
Suppose that $( A , \G , \alpha , \beta )$ is an object crossed system.
Let $\{ u_{\si} \}_{\si \in \G}$ be a copy of $\G$.
We let $A \rtimes^{\alpha}_{\beta} \G$ denote the set of formal sums of the form $\sum_{\si \in \G} a_\si u_\si$
where $a_{\si} \in A_{r(\si)}$, for $\si \in \G$, and $a_{\si} = 0$, for all but finitely many $\si \in \G$.
If $\sum_{\si \in \G} a_\si u_\si$ and $\sum_{\si \in \G} a_\si' u_\si$ are two such formal sums, then their sum is defined to be
$
\sum_{\si \in \G} a_\si u_\si + \sum_{\si \in \G} a_\si' u_\si = \sum_{\si \in \G} (a_\si + a_\si') u_\si.
$
The product of two such formal sums is defined to be the additive extension of the relations
$
a_\si u_\si \cdot a_{\ta}' u_{\ta} = a_\si \alpha_\si ( a_{\ta} ' ) \beta_{\si,\ta} u_{\si \ta},
$
when $(\si,\ta) \in \G_2$, and
$
a_\si u_\si \cdot a_{\ta}' u_{\ta} = 0,
$
otherwise. For all $\si \in \G$ we put $( A \rtimes^{\alpha}_{\beta} \G )_\si = A_{r(\si)} u_\si$.
With this grading, $A \rtimes^{\alpha}_{\beta} \G$ is an object unital $\G$-graded ring which
is called an {\it object crossed product} (for the details, see \cite[Proposition 16]{CNP}). 
In the special case when all of the rings $A_e$ coincide with the same ring $B$,
all $\beta_{\sigma,\tau} = 1_B$ and all $\alpha_{\sigma} = {\rm id}_B$, then 
the corresponding object crossed product equals $B[ \G ]$, the {\it groupoid ring} of $\G$ over $B$. 
\end{ex}

\begin{ex}\label{secondexample}
Let $A$ be a ring and suppose that $\alpha=(A_\sigma,\alpha_\sigma)_{\sigma \in \G}$ is a {\it unital partial action} of $\G$ in $A.$ 
Recall from \cite{bagio2012} that this means that for all $\sigma,\tau \in \G,$
\begin{itemize}

\item $A_\sigma$ is an ideal of $A_{r(\sigma)}$ and $A_{r(\sigma)}$ is an ideal of $A,$;

\item there exists a central idempotent $1_\sigma$ of $A$ such that $A_\sigma=A1_\sigma,$; 

\item $\alpha_\sigma: A_{\sigma\m}\to A_\sigma$ is a ring isomorphism;

\item $\alpha_{\sigma\tau}$ is an extension of $\alpha_\sigma \circ \alpha_\tau,$ provided that $(\sigma, \tau)\in \G_2.$.

\end{itemize} 
The corresponding {\it partial skew groupoid ring} 
$A \star_\af \G$ is defined to be the set of formal sums of the form 
$\sum_{g\in \G}a_g\delta_g$, where $a_g\in A_g$ and $a_g = 0$ for all but finitely many $\sigma \in \G$.
If $\sum_{\si \in \G} a_\si \delta_\si$ and $\sum_{\si \in \G} a_\si' \delta_\si$ are two such formal sums, 
then their sum is defined to be
$$
\sum_{\si \in \G} a_\si \delta_\si + \sum_{\si \in \G} a_\si' \delta_\si = \sum_{\si \in \G} (a_\si + a_\si') \delta_\si.
$$
The product of two such formal sums is defined to be the additive extension of the relations
$
a_\si \delta_\si \cdot a_{\ta}' \delta_{\ta} = a_\sigma \alpha_\sigma( a_\tau' 1_{\sigma^{-1}}) \delta_{\sigma \tau},
$
when $(\si,\ta) \in \G_2$, and
$
a_\si \delta_\si \cdot a_{\ta}' \delta_{\ta} = 0,
$
otherwise. For all $\si \in \G$ we put $( A \star_\af \G )_\si = A_\sigma \delta_\si$.
With this grading $A \star_\af \G$ is an object unital $\G$-graded ring with
$1_{ (A \star_\af \G)_e } = 1_e \delta_e$, for $e \in \G_0$.
\end{ex}

\subsection{Graded modules}

Let $M=\bigoplus_{\si \in \G} M_\si$ be a $\G$-graded left $R$-module. 
%Recall from the introduction that this means that there for all $\si \in \G$ is
%an additive subgroup $M_\si$ of $M$ such that $M = \oplus_{\si \in \G} M_\si$
%and for all $\si,\tau \in \G$ the inclusion $R_{\si} M_{\tau} \subseteq M_{\si \ta}$ holds,
%if $(\si,\ta) \in G_2$, and $R_{\si} M_{\ta} = \{ 0 \}$, otherwise.
If $N=\bigoplus_{\si \in \G} N_\si$ is another left $R$-module graded by $\G$, then a left $R$-module homomorphism
$f : M \to N$ is said to be \emph{graded} if for all $\si \in \G$ the inclusion $f(M_{\si}) \subseteq N_{\si}$ holds.
We let $\Hom_{\GR}(M,N)$ denote the set of graded $R$-module homomorphisms $M \to N$.
The collection of (unitary) $\G$-graded left $R$-modules and the collection of graded homomorphisms
together form a category which we denote by $\grmd$ ($\grmod$).
Analogously, the category $\gmdr$ ($\gmodr$) of (unitary) 
$\G$-graded right $R$-modules is defined.
If $S$ is another $\G$-graded ring, then we say that an
$R$-$S$-bimodule is graded if it is graded both as a left $R$-module
and as a right $S$-module.
The collection of (unitary) $\G$-graded $R$-$S$-modules and the 
collection of graded homomorphisms
together form a category which we denote by $\grmds$ ($\grmods$).
The set $H(M) = \bigcup_{\sigma \in \G} M_\sigma$ is called the set of \emph{homogeneous elements of $M$}.  
If $m \in M_\sigma$ is a non-zero element, 
then we say that $m$ is of \emph {degree $\sigma$} and write $\deg(m)=\sigma$. 
Any $m \in M$ has a unique decomposition $m = \sum_{\sigma \in \G} m_{\sigma}$, where
$m_{\sigma} \in M_{\sigma}$, for $\sigma \in \G$, and all but a finite number of the $m_{\sigma}$ are zero.  
If $N$ is an $R$-submodule of $M$, then it is called a graded submodule if 
$N = \oplus_{\sigma \in \G} (N \cap M_{\sigma})$. 
%In that case, the quotient module $M/N$ is a graded module if we put 
%$(M/N)_\sigma = \{ m_\sigma + N \mid m_\sigma \in M_\sigma \}$, for $\sigma \in \G$.
%An (left, right) ideal of $R$ is called graded if it is graded as a (left, right) submodule of $R$.
In the sequel, we will make repeated use of the following

\begin{prop}\label{unitarymodule}
If $R$ is object unital and $M \in \grmd$ ($M \in \gmdr$), 
then $M \in \grmod$ ($M \in \gmodr$) 
if and only if for all $\sigma \in \G$
and all $m_\sigma \in M_\sigma$, the equality 
$1_{R_{r(\sigma)}}m_\sigma = m_\sigma$ ($m_\sigma = m_\sigma 1_{R_{d(\sigma)}})$ holds.
\end{prop}

\begin{proof}
First we show the ``if'' part. 
Suppose that for all $\sigma \in \G$
and all $m_\sigma \in M_\sigma$, the equality 
$1_{R_{r(\sigma)}}m_\sigma = m_\sigma$ ($m_\sigma = m_\sigma 1_{R_{d(\sigma)}})$ holds.
Take $m \in M$. Then there is $n \in \mathbb{N}$,
$\sigma_1,\ldots,\sigma_n \in \G$ and 
$m_{\sigma_i} \in M_{\sigma_i}$, for $i = 1,\ldots,n$,
such that $m = \sum_{i=1}^n m_{\sigma_i}$.
Let $T = \{ r(\sigma_i) \mid i \in \{ 1,\ldots,n \} \}$
and put $u = \sum_{t \in T} 1_{R_t}$. Then, clearly, $u m = m$,
and hence $M$ is unitary.

Now we show the ``only if'' part.
We will show the part about $\grmod$. The part about $\gmodr$ is analogous and is left to the reader.
Let $M$ be a module in $\grmod$. Take $\sigma \in \G$ and $m_\sigma \in M_\sigma$.
Since $M$ is unitary, it follows that there exist $n \in \mathbb{N}$, $\tau_i , \rho_i \in \G$,
$r_i \in R_{\tau_i}$ and $m_i \in M_{\rho_i}$, for $i = 1,\ldots,n$, such that
$m_\sigma = \sum_{i=1}^n r_i m_i$ and $\tau_i \rho_i = \sigma$, for $i = 1,\ldots,n$.
Since $r(\sigma) = r(\tau_i \rho_i) = r(\tau_i)$, for $i = 1,\ldots,n$, we get that
$1_{R_{r(\sigma)}} m = \sum_{i=1}^n 1_{R_{r(\tau_i)}} r_i m_i = \sum_{i=1}^n r_i m_i = m$.
\end{proof}

\begin{cor}
If $R$ is a unital ring and $M \in \rmd$ ($M \in \mdr$), %is a left (right) $R$-module, 
then $M \in \rmod$ ($M \in \modr$) if and only if for all
$m \in M$ the equality $1_R m = m$ ($m 1_R = m$) holds.
\end{cor}

\begin{proof}
This follows immediately from Proposition \ref{unitarymodule} upon letting $\G$ be the trivial group.
\end{proof}

\subsection{Graded abelian groups}

An additive group $A$ is said to be $\G$-{\it graded} if there for all $\sigma \in \G$
is an additive subgroup $A_\sigma$ of $A$ such that $A = \bigoplus_{\sigma \in \G} A_\sigma$ as additive groups.
If $B$ is another $\G$-graded additive group, then a group homomorphism
$f : A \to B$ is said to be \emph{graded} if for all $\si \in \G$ the inclusion $f(A_{\si}) \subseteq B_{\si}$ holds.
The collection of $\G$-graded additive groups and the collection of graded homomorphisms
together form an abelian category which we denote by $\AbG.$ 
Groups of this type can always, in a natural way, be
viewed as graded left ${\Bbb Z}[\G_0]$-modules, 
since ${\Bbb Z}[\G_0]$ is an object unital graded subring of ${\Bbb Z}[\G]$.  
We call this the {\it trivial grading} of the objects in $\AbG$.

\section{Graded homomorphisms}\label{sec:homomorphisms}

Throughout this section, $R$ denotes a $\G$-graded ring.
%In this section, we introduce suspension for modules in $\grmd$
%and we prove some auxiliary results concerning these
%(see Propositions \ref{prop:suspensionunitary}-\ref{prop:Ggroup}). 
%After that, we introduce the additive group of semi-graded maps
%between objects in $\grmd$ 
%(see Definition \ref{def:semigraded})
%and we prove some 
%structural results for these objects
%(see Prop. \ref{prophom},
%Prop. \ref{prop:finite},
%Prop. \ref{propactions},
%Prop. \ref{propalpha}
%and Cor. \ref{coralpha}).
%At the end of the section, we state and prove a groupoid graded 
%version of a well-known result concerning the 
%hom and tensor functors for bimodules
%(see Proposition \ref{adjunction}).

\begin{defi}
If $M \in \grmd$ and $\sigma \in \G$, % is a $\G$-graded left $R$-module,
then the {\it $\sigma$-suspension} of $M$ is the graded 
additive subgroup $M(\sigma)$ of $M$ defined in the following way.
If $\tau \in \G$, then put:
\[
M(\sigma)_\tau = 
\left\{
\begin{array}{ccc}
M_{\tau \sigma} & \mbox{if} & (\tau,\sigma) \in \G_2 \\
\{ 0 \}         & \mbox{if} & (\tau,\sigma) \notin \G_2.
\end{array}
\right.
\]
\end{defi}

\begin{prop}\label{prop:suspensionunitary}
If $R$ is a (object unital) $\G$-graded ring, 
$M \in \grmd$ ($M \in \grmod$),
then, with the induced left action by $R$, for all $\sigma \in \G$, 
$M(\sigma) \in \grmd$ ($M(\sigma) \in \grmod$). 
\end{prop}

\begin{proof}
First we show that $M \in \grmd$. To this end,
take $\sigma,\tau,\rho \in \G$. We consider two cases.

{ \bf Case 1}:  Suppose $(\rho,\tau) \in \G_2$.\\
{\bf Case 1.1}  $(\tau,\sigma) \in \G_2$.
Then $(\rho\tau,\sigma) \in \G_2$ and so 
$R_\rho M(\sigma)_\tau =
R_\rho M_{\tau\sigma} \subseteq 
M_{\rho \tau \sigma} = 
M(\sigma)_{\rho \tau}$.\\
{\bf Case 1.2}  $(\tau,\sigma) \notin \G_2$.
Then $(\rho\tau,\sigma) \notin \G_2$ and so 
$R_\rho M(\sigma)_\tau =
R_\rho \{ 0 \} =
\{ 0 \} = 
M(\sigma)_{\rho \tau}$.

{\bf Case 2}: Let $(\rho,\tau) \notin \G_2$.\\
{\bf Case 2.1} : $(\tau,\sigma) \in \G_2$.
Then $(\rho,\tau\sigma) \notin \G_2$ and so
$R_\rho M(\sigma)_\tau = 
R_\rho M_{\tau \sigma} = \{ 0 \}$.
{\bf Case 2.2} : $(\tau,\sigma) \notin \G_2$. Then
$R_\rho M(\sigma)_\tau =
R_\rho \{ 0 \} =
\{ 0 \}$.

Now suppose  $R$ object unital and $M \in \grmod$.
Take $\sigma,\tau \in \G$ and $m \in M(\sigma)_\tau$. We wish to show that $1_{R_{r(\tau)}} m  = m$.  If 
$(\tau,\sigma) \in \G_2$. Then $m \in M_{\tau\sigma}$ and therefore $1_{R_{r(\tau)}} m  = m$. 
If $(\sigma,\tau) \notin \G_2$. Then $m = 0$ and so $1_{R_{r(\tau)}} m   = 0 = m$. 
The claim now follows from Proposition \ref{unitarymodule}.
\end{proof}

\begin{prop}\label{composition}
If $M \in \grmd$ and $\sigma,\tau \in \G$, then,
as $\G$-graded additive groups, we have that:
\[
M(\tau)(\sigma) = 
\left\{
\begin{array}{ccc}
M(\sigma\tau) & \mbox{if} & (\sigma,\tau) \in \G_2 \\
\{ 0 \}       & \mbox{if} & (\sigma,\tau) \notin \G_2.
\end{array}
\right.
\]
\end{prop}

\begin{proof}
Take $\sigma,\tau,\rho \in \G$.

{\bf Case 1}: $(\sigma,\tau) \in \G_2$. \\
{\bf Case 1.1}: $(\rho,\sigma) \in \G_2$. Then 
$M(\tau)(\sigma)_\rho = M(\tau)_{\rho\sigma} = M_{\rho\sigma\tau} = M(\sigma\tau)_\rho$.\\
{\bf Case 1.2}: $(\rho,\sigma) \notin \G_2$. Then $(\rho,\sigma\tau) \notin \G_2$ and hence 
$M(\tau)(\sigma)_\rho = \{ 0 \} = M(\sigma\tau)_\rho$.

{\bf Case 2}: $(\sigma,\tau) \notin \G_2$. \\
{\bf Case 2.1}: $(\rho,\sigma) \in \G_2$. Since $(\rho\sigma,\tau) \notin \G_2$, we get that 
$M(\tau)(\sigma)_\rho = M(\tau)_{\rho\sigma} = \{ 0 \}$.\\
{\bf Case 2.2}: $(\rho,\sigma) \notin \G_2$. Then $M(\tau)(\sigma)_\rho = \{ 0 \}$.
\end{proof}

\begin{prop}\label{prop:Ggroup}
If $M \in \grmd$ and $\sigma \in \G$, 
then,  the following assertions hold. 
\begin{itemize} 
\item[(a)] As additive groups, we have that $M(\sigma) = M( d(\sigma) )$.
 \item[(b)] As objects in $\grmd$ the decomposition 
$M = \oplus_{e \in \G_0} M(e)$ holds.
\item[(c)] If $\G$ is a group, then for all $\sigma \in \G$ the equality $M(\sigma) = M$ holds.

\end{itemize}
\end{prop}

\begin{proof}
(a) Take $\tau \in \G$ with $(\tau,\sigma) \in \G_2$.
Then $M(\sigma)_\tau = M(\tau\sigma) = M( \tau\sigma d(\sigma) ) = M( d(\sigma) )_{\tau\sigma}$.
Therefore 
\[
M(\sigma) = \bigoplus_{\tau \in \G, \ (\tau,\sigma) \in \G_2} M(\sigma)_{\tau} = 
\bigoplus_{\tau \in \G, \ (\tau,\sigma) \in \G_2} M(d(\sigma))_{\tau\sigma} \subseteq M(d(\sigma)).
\]
Therefore $M(\sigma) \subseteq M(d(\sigma))$.
Now we show the reversed inclusion. Take $\rho \in \G,$ it is enough to check that $M(d(\sigma))_\rho\subseteq M(\sigma).$ Suppose  $d(\rho)=d(\sigma),$
then $M(d(\sigma))_\rho = M_{\rho d(\sigma)} =
 M_{\rho \sigma^{-1} \sigma} = M(\sigma)_{\rho \sigma^{-1}} \subseteq M(\sigma),$ as desired.

(b) We have
\[ 
M = 
\bigoplus_{\sigma \in \G} M_\sigma = 
\bigoplus_{e \in \G_0} \bigoplus_{\sigma \in \G, d(\sigma)=e} M_\sigma \stackrel{{\rm (a)}}= 
\bigoplus_{e \in \G_0} M(e)
\]
as additive groups. 

(b) This follows from Proposition (a) and (b).
\end{proof}

\begin{defi}\label{def:semigraded}
Suppose that $M,N \in \grmd$.
If $f \in \Hom_R(M,N)$ and $\sigma \in \G$, then we say that 
$f$ is a map of degree $\sigma$ if for all $\tau \in \G$ we have
$f(M_\tau) \subseteq N(\sigma)_\tau.$
We put 
\[
\begin{array}{rcl}
\HOM_R(M,N)_\sigma & = & \{ f \in \Hom_R(M,N) \mid \mbox{$f$ is of degree $\sigma$} \} \\
\HOM_R(M,N) & = & \bigoplus_{\sigma \in \G} \HOM_R(M,N)_{\sigma} \\
\END_R(M)   & = & \HOM_R(M,M).
\end{array}
\]
The elements of $\HOM_R(M,N)$
will from now on be called {\it semi-graded maps}.
\end{defi}

\begin{prop}\label{prophom}
If $M,N,P \in \grmd$ and $\sigma,\tau \in \G$, then:
\begin{itemize}

\item[(a)] ${\rm HOM}_R(M,N)_\sigma = {\rm Hom}_{\GR}(M,N(\sigma))$;

\item[(b)] as additive groups ${\rm Hom}_{\GR}(M(\sigma^{-1}),N)$ is a direct summand in ${\rm HOM}_R(M,N)_\sigma$;

\item[(b)] ${\rm HOM}_R(N,P)_\sigma \circ {\rm HOM}_R(M,N)_\tau \subseteq {\rm HOM}_R(M,P)_{\sigma\tau}$,
if $(\sigma,\tau) \in \G_2$, and ${\rm HOM}_R(N,P)_\sigma \circ {\rm HOM}_R(M,N)_\tau = \{ 0 \}$, otherwise.

\item[(c)] $\Hom_{\GR}(M,N) \supseteq \bigoplus_{e \in \G_0} {\rm HOM}_R(M,N)_e$
with equality if $\G_0$ is finite;

\item[(d)] with the grading ${\rm END}_R(M) = \bigoplus_{\sigma \in \G} {\rm HOM}_R(M,M)_{\sigma}$, ${\rm END}_R(M)$ is a $\G$-graded ring.

\end{itemize}
\end{prop}

\begin{proof}
(a) This follows immediately from the definition of the suspension.

(b) Define the additive map 
\[
p : \Hom_{\GR}(M,N(\sigma)) \to \Hom_{\GR}(M(\sigma^{-1}),N)
\]
by restriction. Let us show that $p$ is well defined.
Take $f \in \Hom_{\GR}(M,N(\sigma))$ and $\tau \in \G$.

\noindent {\bf Case 1}: $d(\sigma) = d(\tau)$. Then 
$f( M(\sigma^{-1})_\tau ) = 
f( M_{\tau \sigma^{-1}} ) \subseteq
M_{\tau \sigma^{-1} \sigma} = M_\tau$.\\
{\bf Case 2: } $d(\sigma) \neq d(\tau)$. Then
$f( M(\sigma^{-1})_\tau ) = 
f( \{ 0 \} ) = \{ 0 \} \subseteq M_\tau$.

Thus $p(f) \in \Hom_{\GR}(M(\sigma^{-1}),N)$.
Define the additive map
\[
i : \Hom_{\GR}(M(\sigma^{-1}),N) \to \Hom_{\GR}(M,N(\sigma))
\]
in the following way. Given $g \in \Hom_{\GR}(M(\sigma^{-1}),N)$ and $\tau \in \G$,
we put $i(g)|_{M(\sigma^{-1})} = g$ and $i(g)|_{M_{\tau}} = \{ 0 \}$, if $r(\sigma) \neq d(\tau)$.
Now we show that $i$ is well defined.

%Take $g \in \Hom_{\GR}(M(\sigma^{-1}),N)$ and $\tau \in \G$.
\noindent {\bf Case 1}: $r(\sigma) = d(\tau)$. Then
$i(g)( M_\tau ) = 
i(g)( M_{\tau \sigma \sigma^{-1}} ) =
i(g)( M(\sigma^{-1} )_{\tau\sigma} ) =
g( M(\sigma^{-1} )_{\tau\sigma} ) \subseteq 
N_{\tau\sigma}= N(\sigma)_\tau.$ \\
{\bf Case 2: }$r(\sigma) \neq d(\tau)$. Then
$i(g)(M_\tau) = \{ 0 \} \subseteq N(\sigma)_\tau$.
It is clear that $p \circ i = \id_{\Hom_{\GR}(M(\sigma^{-1}),N)}$.

(c) Take $f \in \HOM_R(N,P)_\sigma$ and $g \in \HOM_R(M,N)_\tau$.
Take $\rho \in \G$. 
Then we get that
$(f \circ g)(M_\rho) = 
f( g (M_\rho) ) \subseteq
f( M(\tau)_\rho ) \subseteq 
N(\tau)(\sigma)_\rho$. 

\noindent {\bf Case 1}:$(\sigma,\tau) \in \G_2$.  Now Proposition \ref{composition} gives 
$(f \circ g)(M_\rho) \subseteq
N(\sigma\tau)_\rho$. Thus $f \circ g \in \HOM_R(M,P)_{\sigma\tau}$.\\
{\bf Case 2: }$(\sigma,\tau) \notin \G_2$. Proposition \ref{composition} again implies that
$(f \circ g)(M_\rho) = \{ 0 \}$. Thus $f \circ g = 0$.

(d) From (a) we get that 
\[
\bigoplus_{e \in \G_0} \HOM_R(M,N)_e = 
\bigoplus_{e \in \G_0} \HOM_{\GR} (M,N(e)) \subseteq
\Hom_{\GR}(M,N).
\]
Suppose that $\G_0$ is finite. For all $e \in \G_0$ let 
$
p_e : N \to N(e) 
$
be the projection.
Then $p_e \in \Hom_{\GR}(N,N(e))$ and $\sum_{e \in \G_0} p_e = \id_N$.
Take $f \in \Hom_{\GR}(M,N)$. Then, for all $e \in \G_0$, $p_e \circ f \in \Hom_{\GR}(M,N(e))$,
and thus 
$
f = \id_N \circ f = \sum\limits_{e \in \G_0} p_e \circ f \in \bigoplus\limits_{e \in \G_0} \Hom_{\GR}(M,N(e)).
$
%(e) follows from (c)
\end{proof}

\begin{rem}
If we let $\G$ be a group in Proposition \ref{prophom}, then 
clearly the kernel of the map $p$ equals zero. Thus, we retrieve the equality 
\[
\Hom_{\GR}(M,N(\sigma)) = \Hom_{\GR}(M(\sigma^{-1}),N)
\]
from \cite[p. 25]{nastasescu2004}.
\end{rem}

\begin{prop}\label{prop:finite}
Suppose that $M,N \in \grmd$. 
\begin{itemize}

\item[(a)] The inclusion ${\rm HOM}_R(M,N) \subseteq {\rm Hom}_R(M,N)$ holds.

\item[(b)] If $R$ is object unital and $M,N \in \grmod$, then 
\[
{\rm HOM}_R(M,N) = {\rm Hom}_R(M,N)
\]
holds if $\G$ is finite or $M$ is finitely generated as an object in 
$\rmod$.
\end{itemize}
\end{prop}

\begin{proof}
(a) This is trivial. Now we show (b).
Suppose first that $\G$ is finite.
Take $f \in \Hom_R(M,N)$ and $\sigma \in \G$.
From Proposition \ref{unitarymodule}, it follows that $1_{r(\sigma)} M_\sigma = M_\sigma$.
Hence $f(M_\sigma) = f( 1_{r(\sigma)} M_\sigma ) = 1_{r(\sigma)} f( M_\sigma )$.
Thus $$f(M_\sigma) \subseteq 1_{r(\sigma)} N = \sum_{ \stackrel{\alpha \in \G}{r(\alpha)=r(\sigma)}} N_\alpha =
\sum_{\stackrel{\alpha \in \G}{r(\alpha)=r(\sigma)}} N_{\sigma \sigma^{-1} \alpha} =
\sum_{\stackrel{\tau \in \G}{r(\tau)=d(\sigma)}} N_{ \sigma \tau }.$$ 
Take $\tau \in \G$. Now we define $f_\tau \in \HOM_R(M,N)_\tau$ in the following way.
Take $\sigma \in \G$ and $m_\sigma \in M_\sigma$. 
If $r(\tau) \neq d(\sigma)$, then put $f_\tau(m_\sigma)=0$.
If $r(\tau)=d(\sigma)$, then let $f_\tau(m_\sigma)$ be the component of $f(m_\sigma)$ of degree $\sigma\tau$.
It is clear that $f = \sum_{\tau \in \G} f_\tau \in \HOM_R(M,N)$. 

Now suppose that $M$ is finitely generated as an object in $\rmod$
Take $p \in \mathbb{N}$, $\sigma_1 , \ldots , \sigma_p \in \G$ and
non-zero $m_i \in M_{\sigma_i}$, for $i =1,\ldots,p$, such that $M = \sum_{i=1}^p R m_i$.
Take $i \in \{ 1,\ldots,p \}$. There is a finite subset $S_i$ of $\G$ such that 
$f(m_i) \in \sum_{\sigma \in S_i} N_\sigma$.
From Proposition \ref{unitarymodule}, it follows that $1_{r(\sigma_i)} m_i = m_i$. 
Thus $f(m_i)  = 1_{r(\sigma_i)} f(m_i)$ and  
we can therefore assume that $r(\sigma_i) = r(\sigma)$ for all $\sigma \in S_i$.
Put $T_i = \sigma_i^{-1} S_i$. This is a well defined set since
$d( \sigma_i^{-1} ) = r(\sigma_i) = r(\sigma)$ for $\sigma \in S_i$.
Let $T = \bigcup_{i=1}^p T_i$.
Take $\sigma \in \G$. Then $M_\sigma = \left( \sum\limits_{i=1}^p R m_i \right)_\sigma = 
\sum\limits_{i, d(\sigma_i)=d(\sigma) } R_{\sigma \sigma_i^{-1}} m_i$.
Write $f(m_i) = \sum\limits_{j=1}^{q_{ij}} n_{ij}$ for some $n_{ij} \in N_{\alpha_{ij}}$
where $\alpha_{ij} \in \G$. Therefore $f( M_\sigma ) = 
\sum\limits_{i, \ d(\sigma_i)=d(\sigma) } R_{\sigma \sigma_i^{-1}} f(m_i) = 
\sum\limits_{i, \ d(\sigma_i)=d(\sigma) } \sum\limits_{j=1}^{q_{ij}} R_{\sigma \sigma_i^{-1}} n_{ij}$.
Put $S_\sigma = \{ \sigma_i^{-1} \alpha_{ij} \mid d(\sigma)=d(\sigma_i), \ r(\sigma_i)=r(\alpha_{ij}) \}.$
The above calculation shows that $f( M_\sigma ) \subseteq \sum_{\tau \in S_\sigma} N_{\sigma \tau}$.
Take $\tau \in \G$. Now we define $f_\tau \in \HOM_R(M,N)_\tau$ in the following way.
Take $\sigma \in \G$ and $m_\sigma \in M_\sigma$. 
If $\tau \notin S_\sigma$, then put $f_\tau(m_\sigma)=0$.
If $\tau \in S_\sigma$, then let $f_\tau(m_\sigma)$ be the component of $f(m_\sigma)$ of degree $\sigma\tau$
in $\sum_{\tau \in S_\sigma} N_{\sigma \tau}$.
It is clear that $f = \sum_{\tau \in \G} f_\tau \in \HOM_R(M,N)$. 
\end{proof}

\begin{rem}
If $\G$ is infinite or if $M$ is not finitely generated, then the 
equality $\HOM_R(M,N) = \Hom_R(M,N)$ does not always hold
(see \cite[p. 11]{N82} for a counterexample in the case when $\G$ is a group). 
\end{rem}

Now we gather groupoid graded versions of some results concerning 
bimodule actions on groups of homomorphisms (see e.g. \cite[p. 78]{R79}
for the ungraded situation).

\begin{prop}\label{propactions}
Suppose that $R$ and $S$ are $\G$-graded rings.
\begin{itemize}

\item[{\rm (a)}] If $M \in \grmds$ (with $S$ object unital and
$M \in \gmods$) and $N \in \grmd$,
then ${\rm HOM}_R(M,N) \in \gsmd$ 
(${\rm HOM}_R(M,N) \in \gsmod$),
where  $$(s f)(m) = f(ms)$$, for $s \in S$, 
$f \in {\rm HOM}_R(M,N)$ and $m \in M$.

\item[{\rm (b)}] If $M \in \grmds$
(with $R$ object unital and $M \in \grmod$)  
and $N \in \gmds$,
then ${\rm HOM}_S(M,N) \in \gmdr$ 
(${\rm HOM}_S(M,N) \in \gmodr$), 
where $$(f r)(m) = f(rm),$$ 
for $f \in {\rm HOM}_R(M,N)$, $r \in R$ and $m \in M$.

\item[{\rm (c)}] If
$M \in \gmds$ and 
$N \in \gmds$ 
(with $R$ object unital and $N \in \grmod$),
then ${\rm HOM}_S(M,N) \in \grmd$ (${\rm HOM}_S(M,N) \in \grmod$), 
where $$(r  f)(m) = rf(m),$$ for $r \in R$, 
$f \in{\rm HOM}_R(M,N)$, and $m \in M$.

\item[{\rm (d)}] If
$M \in \grmd$ and 
$N \in \grmds$ 
(with $S$ object unital and $N \in \gmods$),
then ${\rm HOM}_R(M,N) \in \gmds$ (${\rm HOM}_R(M,N) \in \gmods$),
where $$(f s)(m) = f(m)s,$$ for $f \in{\rm HOM}_R(M,N)$, 
$s \in S$ and $m \in M$.
\end{itemize}
\end{prop}

\begin{proof}
We only show (a). The rest of the statements are shown in a 
similar fashion and are therefore left to the reader.
It is clear that the action of $S$ defines a left $S$-module
structure on $\HOM_R(M,N)$. What is left to check are the statements
concerning the grading. To this end,
take $\sigma,\tau,\rho \in \G$ and %, $s_\sigma \in S_\sigma$,
$f_\tau \in \HOM_R(M,N)_\tau$. % and $m_\rho \in M_\rho$.

\noindent {\bf Case 1}: $(\sigma,\tau) \in \G_2$.\\
{\bf Case 1.1}: $(\rho,\sigma) \in \G_2$.
Then $(S_\sigma f_\tau)(M_\rho) = 
f_\tau(M_\rho S_\sigma) \subseteq 
f_\tau( M_{\rho\sigma}) \subseteq 
M(\tau)_{\rho\sigma} = 
M(\tau)(\sigma)_\rho =
M(\sigma\tau)_\rho$, by Proposition \ref{composition}.\\
{\bf Case 1.2}: $(\rho,\sigma) \notin \G_2$.
Then 
$(S_\sigma f_\tau)(M_\rho) = 
f_\tau(M_\rho S_\sigma) = 
\{ 0 \} = 
M(\sigma\tau)_\rho$. 
From Proposition \ref{prophom}(a) it now follows that
$S_\sigma f_\tau \subseteq \HOM_R(M,N)_{\sigma\tau}$.

\noindent {\bf Case 2}: $(\sigma,\tau) \notin \G_2$.
\\
{\bf Case 2.1}:$(\rho,\sigma) \in \G_2$.
Then $(S_\sigma f_\tau)(M_\rho) = 
f_\tau(M_\rho S_\sigma) \subseteq 
f_\tau( M_{\rho\sigma}) \subseteq 
M(\tau)_{\rho\sigma} = \{ 0 \}$.
\\
{\bf Case 2.2}:$(\rho,\sigma) \notin \G_2$.
Then 
$(S_\sigma f_\tau)(M_\rho) = 
f_\tau(M_\rho S_\sigma) = 
\{ 0 \}$. 
Therefore, $S_\sigma f_\tau = \{ 0 \}$.

Now suppose that $S$ is object unital and that 
$M_S$ is unitary. We wish to use Proposition \ref{unitarymodule} 
to show that $\HOM_R(M,N)$ is unitary as a left $S$-module.
Take $m_\rho \in M_\rho$. 
Case 1: $(\rho,\tau) \in \G_2$. Then
$( 1_{R_{r(\tau)}} f_\tau ){m_\rho} = 
f_\tau( m_\rho 1_{R_{r(\tau)}} ) =
f_\tau(m_\rho)$.
Case 2: $(\rho,\tau) \notin \G_2$. Then
$( 1_{R_{r(\tau)}} f_\tau ){m_\rho} = 
f_\tau( m_\rho 1_{R_{r(\tau)}} ) =
f_\tau( 0 ) = 0 = f_\tau(m_\rho)$.
Therefore $1_{R_{r(\tau)}} f_\tau = f_\tau$.
%The claim now follows from Proposition \ref{unitarymodule}.
\end{proof}

\begin{prop}\label{propalpha}
Suppose that $R$ is an object unital ring and $M \in \grmod$. 
If we equip ${\rm HOM}_R(R,M)$ with the $\grmod$
structure defined in Proposition \ref{propactions}(a), then the map 
$
\alpha : M \to {\rm HOM}_R(R,M), 
$
defined by $\alpha(m)(r) = rm$, for $m \in M$ and $r \in R$,
is an isomorphism in $\grmod$.
\end{prop}

\begin{proof}
Clearly $\alpha$ is an $R$-linear graded map.
Define a map 
\[
\beta : \HOM_R(R,M) \to M
\]
in the following way.
Take $f \in \HOM_R(R,M)$.
Take $\sigma_1,\ldots,\sigma_n \in \G$
such that $f = \sum_{i=1}^n f_{\sigma_i}$ and
$f_{\sigma_i} \in \HOM_R(R,M)_{\sigma_i}$ for $i = 1,\ldots,n$.
Put $\beta( f ) = \sum_{i=1}^n f_{\sigma_i}( 1_{r(\sigma_i)} )$.
Clearly $\beta$ is additive.
Also, since 
$f_{\sigma_i}( 1_{r(\sigma_i)} ) \in M_{ r(\sigma_i) \sigma_i } =
M_{\sigma_i}$, we get that $\beta$ is graded.
Now we show that $\beta$ respects left multiplication by 
homogeneous elements from $R$. To this end, take $\tau,\sigma \in G$,
$r_\tau \in R_\tau$ and $f_\sigma \in HOM_R(R,M)_\sigma$.

 {\bf Case 1}: : $d(\tau) \neq r(\sigma)$.
Then $\beta( r_\tau \cdot f_\sigma ) = 
\beta( 0 ) = 
0 =
r_\tau f_\sigma( 1_{r(\sigma)} ) = r_\tau \beta( f_\sigma )$.

{\bf Case 2}: $d(\tau) = r(\sigma)$.
Then $\beta( r_\tau \cdot f_\sigma ) =
\beta( f_\sigma(  r_\tau )) = 
f_\sigma( 1_{r(\tau\sigma)} r_\tau ) =
f_\sigma( 1_{r(\tau)} r_\tau ) =
f_\sigma( r_\tau ) =
f_\sigma( r_\tau 1_{d(\tau)} ) =
f_\sigma( r_\tau 1_{r(\sigma)} ) =
r_\tau f_\sigma( 1_{r(\sigma)} ) =
r_\tau \beta(f_\sigma)$.

Finally we show that $\beta \circ \alpha = {\rm id}_M$
and $\alpha \circ \beta = {\rm id}_{{\rm HOM}_R(R,M)}$.
Take $\sigma \in \G$ and $m_\sigma \in M_\sigma$. Then 
$\beta( \alpha(m_\sigma) ) = 
\alpha(m_\sigma)(1_{r(\sigma)} )=
1_{r(\sigma)} m_\sigma =
m_\sigma$.  
Next, take $(\tau,\sigma) \in \G_2$, $f_\sigma \in {\rm HOM}_R(R,M)_\sigma$
and $r_\tau \in R_\tau$. Then
$\alpha( \beta( f_\sigma ) )(r_\tau) =
\alpha( f_\sigma( 1_{r(\sigma)} ) )(r_\tau) =
r_\tau f_\sigma( 1_{r(\sigma)} ) =
f_\sigma( r_\tau 1_{r(\sigma)} ) =
f_\sigma( r_\tau 1_{d(\tau)} ) =
f_\sigma( r_\tau )$.
\end{proof}

\begin{rem}\label{remalpha}
Suppose that $R$ is an object unital $\G$-graded ring and 
$M \in \gmodr$. % is a unitary $\G$-graded right $R$-module.
If we equip $\HOM_R(R,M)$ with the $\gmodr$
structure defined in 
Proposition \ref{propactions}(b), then, in a fashion similar
to the proof of Proposition \ref{propalpha}, one can
show that the map 
$
\alpha' : M \to HOM_R(R,M), 
$
defined by $\alpha'(m)(r) = mr$, for $m \in M$ and $r \in R$,
is an isomorphism in $\gmodr$.
\end{rem}

\begin{cor}\label{coralpha}
If $R$ is an object unital $\G$-graded ring, then the maps
\[
\alpha : R \to \END_R(R)
\]
and
\[
\alpha' : R \to \END_R(R),
\]
from Proposition \ref{propalpha} and Remark \ref{remalpha}
are isomorphisms in $\grmodr$.
%$\G$-$R$-$R$ bimodule isomorphisms. 
\end{cor}

The proofs of the following two proposition is 
similar to the ungraded case (found e.g. in \cite{R79}). 

\begin{prop}\label{prop:directsum} The following assertions hold.
\begin{itemize}
\item [(a)]
If  $M$ , $\{N_i \}_{i \in I}, P \in \grmd$,
then the isomorphism 
\[
{\rm HOM}_R(\oplus_{i \in I} N_i , M) \cong \oplus_{i \in I} {\rm HOM}_R(N_i,M)
\] holds in $\AbG$.
\item [(b)]
If 
$
M \rightarrow N \rightarrow P \rightarrow 0
$
is an exact sequence in $\grmd$,
then the induced sequence in $\AbG$:
\[
0 \rightarrow {\rm HOM}_R(P,Q) \rightarrow {\rm HOM}_R(N,Q) \rightarrow {\rm HOM}_R(M,Q)
\]
is exact.
\end{itemize}

\end{prop}

\begin{defi}
If $M$ is a $\G$-graded right $R$-module and $N$ is a $\G$-graded left
$R$-module, then we may consider $M \otimes_R N$ as an object in
$\AbG$, where the grading is defined by letting 
$(M \otimes_R N)_{\sigma}$, $\sigma \in \Gamma$, be the ${\Bbb Z}$-module 
generated by all $m_{\tau} \otimes n_{\rho}$, $d(\tau)
= r(\rho)$, $\tau \rho = \sigma$, $m_{\tau} \in M_{\tau}$,
$n_{\rho} \in N_{\rho}$. To see that this is well defined, note
that $M \otimes_{R} N=M \otimes_{\Bbb Z} N / L$ where $L$ is the graded subgroup
of $M \otimes_{\Bbb Z} N$ generated by the elements of the form $mr
\otimes n - m \otimes rn$. The grading on $M \otimes_R N$ is
therefore induced by the grading on $M \otimes_{\Bbb Z} N$.
If $S$ is another $\G$-graded ring and $N$ is a
$\G$-graded $R$-$S$-bimodule, then if we 
consider $M \otimes_R N$ with it's right $S$-module structure
it is a $\G$-graded right $S$-module.
\end{defi}

The usual relation between $\Hom$ and $\otimes$ carry over
to the $\G$-graded situation:

\begin{prop}\label{adjunction}
If $R$ and $S$ are $\G$-graded rings,
$M \in \gmdr$, %is a $\G$-graded right $R$-module, 
$N \in \grmds$ 
and $P \in \gmds$, % is a $\G$-graded right $S$-module,
then the map
\[
\varphi : HOM_S( M \otimes_R N , P) \to HOM_R(M , HOM_S(N , P))
\]
defined by $\varphi(f)(m)(n) = f(m \otimes n)$,
for $f \in \HOM_R(M \otimes_R N)$, $m \in M$ and $n \in N$,
is a well defined isomorphism in $\AbG$.
In that case, if $R$ and $S$ are object unital, 
$M \in \gmodr$, 
$N \in \grmods$ and 
$P \in \gmods$, then 
the functors 
$$M \otimes_R - : \grmod \to \gmods$$ 
and 
$$HOM_S(-,P) : \gmods \to \grmod$$ 
form and adjoint pair.
\end{prop}

\begin{proof}
From Proposition \ref{propactions}(b) it follows that 
$\Hom_S(N,P)$ is a (unitary) $\G$-graded right $R$-module
(if $N$ is unitary as a left $R$-module).
It is clear that $\varphi$ is graded.
To show that $\varphi$ is an isomorphism we proceed as in the
classical case (see e.g. \cite[p. 92]{R79}). 
To prove the last statement, let 
$\phi$ denote the restriction of $\varphi$ to the sum
of the components of degree in $\G_0$. Then we get that
\[
\phi : 
\oplus_{e \in \G_0} HOM_S( M \otimes_R N , P)_e \to 
\oplus_{e \in \G_0} HOM_R(M , HOM_S(N , P))_e.
\]
is an isomorphism, or, in other words, that 
\[
\phi : 
HOM_{\gmods} ( M \otimes_R N , P) \to 
HOM_{\grmod} (M , HOM_S(N , P))
\]
is an isomorphism.
\end{proof}

%
%We end this section by remarking that the functor $U$ has a right
%adjoint $$G : R{\rm -}mod \rightarrow R{\rm -}gr$$ which to a
%left $R$-module $M$ associates $G(M) = \oplus_{\sigma \in \Gamma}
%{ }^{\sigma}M$, where ${ }^{\sigma}M = \{ { }^{\sigma}x \mid x
%\in M \}$, with an $R$-module structure defined by
%$$
%\left\lbrace
%\begin{array}{l}
%{ }^{\tau}x + { }^{\tau}y = { }^{\tau}(x+y) \\
%r \cdot { }^{\tau}x = \sum_{\sigma \in \Gamma, \
%d(\sigma)=r(\tau)} { }^{\sigma\tau}(r_{\sigma}x),
%\end{array}
%\right.
%$$
%$\tau \in \Gamma$, $x,y \in M$, $r \in R$. If $f : M \rightarrow
%N$ is $R$-linear, then $G(f) : G(M) \rightarrow G(N)$ is defined
%by $G(f)({ }^{\sigma}x) = { }^{\sigma}f(x)$, $\sigma \in \Gamma$,
%$x \in M$. It is easy to check that $G$ is exact.
%
%$ $

\section{The ungrading functor}\label{sec:theungradingfunctor}

Throughout this section, $R$ denotes a $\G$-graded ring.
In this section, we analyse properties of the ungrading functor
$U$ from $\grmd$ (or $\grmod$) to $\rmd$ (or $\rmod$),
including the properties
direct summand, free, finitely generated, finitely presented, projective, injective,
essential, small and flat.
Note that a large part of the results that we are about to present
have already been obtained in the unital situation 
(see \cite{nastasescu2004} for the group graded case and
\cite{lundstrom2004} for the groupoid graded situation).
Therefore, we will often just sketch the proofs or refer
to the existing proofs in the unital case.
The following lemma will be used in the sequel.

\begin{lem}\label{lem:GRADEDLEMMA}
Suppose that $M,N,P \in \grmd$ and
$f : M \to P$, 
$g : N \to P$ and
$h : M \to N$ are $R$-linear maps such that 
$f = g \circ h$. If $f$ and $g$ ($f$ and $h$) 
are graded, then there is a graded map 
$h' : M \to N$ ($g' : N \to P$) such that 
$f = g \circ h'$ ($f = g' \circ h$).
\end{lem}

\begin{proof}
The proof of \cite[3.1.1 Lemma]{lundstrom2004} works in 
the non-unital situation also.
\end{proof}

Let $A$ and $B$ be objects in an abelian category. Recall that
$B$ is called a direct summand of $A$ if there is an object $C$
in the category such that $A \cong B \oplus C$.
%As a consequence from Lemma \ref{lem:GRADEDLEMMA} 
%we have the next result.

\begin{cor}\label{cor:DIRECTSUMMAND}
Suppose that $M,N \in \grmd$ ($M,N \in \grmod$ and $R$ is object unital). 
If $N$ is a graded submodule of $M$, then $N$ is a direct summand of $M$
in $\grmd$ ($\grmod$) if and only if $U(N)$ is a direct summand of 
$U(M)$ in $\rmd$ ($\rmod$).
\end{cor}

\begin{proof}
This follows from Lemma \ref{lem:GRADEDLEMMA}.
\end{proof}

%Moreover, the proof is similar to the ungraded case.
%
%\begin{lem}
%$N$ is a direct summand of $M\in\grumod$ if, and only if, there exists $f:M\to N$ in $\grumod$ such that $\restr{f}{N}=\mathbf{1}_N$. In particular, $M=N\oplus\operatorname{Ker} f$.
%\end{lem}

\begin{defi}\label{def:split}
Suppose that $L,M,N \in \grmd$ ($L,M,N \in \grmod$ and $R$ is object unital). 
We say that a short exact sequence of graded maps:
\begin{equation}\label{shortexact}
\begin{tikzcd}[column sep = small]
0 \arrow[r] & L \arrow[r, "f"] & M \arrow[r, "g"] & N \arrow[r] & 0
\end{tikzcd}
\end{equation}
{\it splits} if there is a graded isomorphism 
$h : M \to L \,\oplus\, N$ making the following diagram commutative:
\begin{equation}\label{diagcinde}
\begin{tikzcd}[row sep = large, column sep = normal]
0 \arrow[r] & L \arrow[d, "\id_{L}", swap] \arrow[r, "f"] & M \arrow[d, dashed, "h"] \arrow[r, "g"] & N \arrow[d, "\id_{N}"] \arrow[r] & 0 \\
0 \arrow[r] & L \arrow[r, hook, "\iota_L", swap] & L\oplus N \arrow[r, two heads, "\pi_N", swap] & N \arrow[r] & 0
\end{tikzcd}
\end{equation}
\end{defi}

\begin{prop}\label{grcinde}
With the notation from Definition \ref{def:split},
the following conditions are equivalent:
\begin{itemize}
\item[(i)] The sequence (\ref{shortexact}) splits.
\item[(ii)] There exists a graded map $\varphi \colon M \to L$ 
such that $\varphi \circ f = \id_L$.
\item[(iii)] There exists a graded map  $\psi \colon N \to M$ 
such that $g \circ \psi = \id_N$.
\end{itemize}
\end{prop}

\begin{proof} The ungraded proof of the implication (i)$\Rightarrow$(ii) 
and the equivalence (i) $\Leftrightarrow$ (iii) (see e.g. \cite{R79}) 
carry over to the graded situation, 
taking into account Lemma \ref{lem:GRADEDLEMMA}.
The implication (ii)$\Rightarrow$(i) can be proved in a similar way, 
using the Five Lemma for abelian categories (see e.g. \cite[Theorem 5.9]{DB}).
\end{proof}

%In the particular case when $\G$ is a group, it can be proven (see \cite[p.~33]{RH}) that the following conditions are equivalent for $M\in\grmod$:

%\begin{itemize}
%\item[i)] $M$ is free;
%\item[ii)] $M$ has a $R$-basis of homogeneous elements (of  not necessarily %distinct degrees).
%\end{itemize}

%But the converse holds. Indeed,  if $\mathcal{B}$ is a $R$-basis consisting of homogeneous elements of a module $M\in\grmod$, we define a $R$-linear map from $\bigoplus_{m \in \mathcal{B}} R(\sigma_m)$ to $M=\bigoplus_{m\in \mathcal{B}} Rm$, where $\sigma_m=\deg(m)^{-1}\in\G$, sending every $e_m$ into $m\in\mathcal{B}$, where $e_m$ is the element whose entries are all zero except at the $m$-th coordinate where it takes the value $1_{R_{d(\sigma_m)}}$. By Lemma \ref{suspension-ideal} this is a well-defined graded surjective map and since the elements of the basis has trivial annihilators, the map is an isomorphism.

\begin{defi}
If $M \in \grmd$ ($R$ is object unital and $M \in \grmod$), 
then we say that $M$ is {\it (finite) free by suspension}
if there is a (finite) set $I$ and $\sigma_i \in \G$, 
for $i \in I$, such that $M \cong \bigoplus_{i\in I} R(\sigma_i)$ 
in $\grmd$ ($\grmod$).
\end{defi}

\begin{rem}
In \cite{lundstrom2004} the concept ``free by suspension'' 
is called just ``free''. However, since not all groupoid graded
modules that are free by suspension are free in the usual 
module theoretic sense (see \cite[Example 3.2.1]{lundstrom2004}), 
that is, that the module has a ``basis''
with the property that none of the elements of this basis
can be annihilated by non-zero action of the ring, 
we have chosen to introduce this new adjective,
in order to not confuse the reader.
\end{rem}

\begin{prop}
Let $R$ be a $\G$-graded ring.
\begin{itemize}

\item[(a)] The ring $R$ is free by suspension.

\item[(b)] If $R$ has the property that
$R_e = \{ 0 \}$, for all but finitely many $e \in \G_0$,
then $R$ is finite free by suspension.

\item[(c)] If $R$ is unital, then $R$ is finite free by suspension.

\end{itemize}
\end{prop}

\begin{proof}
From Proposition \ref{prop:Ggroup} it follows that
$R = \oplus_{e \in \G_0} R(e)$ as $\G$-graded rings, if
we consider $R$ as a $\G$-graded module over itself.
Therefore (a) holds. (b) follows from the proof of (a);
(c) follows from (b) and \cite[Proposition 2.1.1]{lundstrom2004}.
\end{proof}

%\begin{prop}\label{cocientelibre}
%Every $M\in\grmod$ is the quotient of a module
%which is free by suspension.
%\end{prop}

%\begin{proof}
%Let $\{m_i\colon i\in I\}$ be a homogeneous generator set of $M$, with $%\deg(m_i)=\sigma_i$, $i\in I$. Put $N=\bigoplus_{i\in I} R(\sigma_i^{-1})$. The function $\varphi\colon N\to M$ defined by $R$-linear extension of $e_i\mapsto m_i$, $i\in I$, is an epimorphism in $\grumod$ and, therefore, $M\cong N/\operatorname{Ker}\varphi$ and the isomorphism is in $\grumod$.
%\end{proof}

\begin{prop}\label{prop:DIRECTSUMMANDFREE}
If $M \in \grmd$ is (finite) free by suspension, then
there is $M' \in \grmd$ such that $M$ is (finite) free
by suspension and $U(M \oplus M')$ is weakly free (of finite type).
\end{prop}

\begin{proof}
We use the proof of \cite[Proposition 3.2.2.]{lundstrom2004}.
It is enough to prove the claim when $M = R(\sigma)$ for 
some $\sigma \in \G$. 
Put $M' = \bigoplus\limits_{e \in \G_0 \setminus \{ d(\sigma) \}} R(e)$.
Then $U(M \oplus M') = U(R)$ which is weakly free of finite type.
\end{proof}

\begin{defi}
Let $M \in \grmd$. If $n$ is a non-negative
integer, then we say that $M$ has a (finite) presentation of
length $n$ if there is an exact sequence
$F_n \rightarrow F_{n-1} \rightarrow \cdots
\rightarrow F_0 \rightarrow M \rightarrow 0$ 
of maps and modules in $\grmd$ which all are (finite) free by suspension.
If $M$ has a (finite) presentation of length $0$,
then we say that $M$ is {\it (finitely) generated}.
If $M$ has a (finite) presentation of length 1, 
then we say that $M$ is {\it (finitely) presented}.
\end{defi}

\begin{prop}\label{prop:finitegeneration}
Suppose that $M \in \grmd$. %Then:
\begin{itemize}

\item[(a)] If $M$ has a finite presentation of length $n$,
then $U(M)$ has a finite presentation of length $n$.

\item[(b)] The module $M$ is (finitely) generated if and only if
$U(M)$ is finitely generated.

\end{itemize}
\end{prop}

\begin{proof}
Using Proposition \ref{prop:DIRECTSUMMANDFREE} it is clear that
the proof of \cite[Proposition 3.3.1]{lundstrom2004} works
in the non-unital situation also.
\end{proof}

\begin{prop}\label{prop:presentation}
Suppose that $M \in \grmd$ and $n \in \mathbb{N}$.
\begin{itemize}

\item[(a)] The module $M$ admits a presentation of length $n$.

\item[(b)] There is $F \in \grmd$, with $F$ free by suspension, 
and a  graded submodule $K$ of $F$ such that $F/K$ and $M$ are
isomorphic in $\grmd$. 

\item[(c)] The module $M$ is presented.

\item[(d)] The module $M$ is the direct limit of a direct system
of graded maps and finitely presented graded modules.

\end{itemize}
\end{prop}

\begin{proof}
The proof of \cite[Proposition 3.3.4]{lundstrom2004}
works in the non-unital situation also.
\end{proof}

Recall that an object $A$ in an abelian category
${\mathcal A}$ is called projective if the functor
$\hom_{\mathcal{A}}(A,-) : {\mathcal A} \to \Ab$ is exact.
Similarly to the ungraded situation, there is a 
characterization of projective objects in $\grmd$: 

\begin{prop}\label{proyhomex}
A module $P \in\grmd$ is projective if and only if 
for every surjective $g : M \to N$ in $\grmd$
and every $h : P \to N$ in $\grmd$, there exists
$\overline{h} : P \to M$ in $\grmod$
such that $h=g\circ\overline{h}$.
\end{prop}

\begin{prop}\label{prop:ABELIANPROJECTIVE}
Let ${\mathcal A}$ be an abelian category. Then:
\begin{itemize}
\item[{\rm (i)}] If $(P_i)_{i \in I}$ is a family
of objects in ${\mathcal A}$, then $\bigoplus_{i \in I} P_i$ is
projective if and only if each $P_i$ is projective.
\item[{\rm (ii)}] If
$0 \rightarrow A \rightarrow B \stackrel{\alpha}{\rightarrow}C
\rightarrow 0$ is an exact sequence in ${\mathcal A}$,
then the sequence splits if and only if there is
$\beta : C \rightarrow B$ such that
$\alpha \circ \beta = {\rm id}_C$.
\end{itemize}
\end{prop}

\begin{proof}
These are standard facts which can be found in \cite{S75}.
\end{proof}

\begin{lem}\label{lem:FREEIMPLPROJECTIVE}
If $M \in \grmd$ is free by suspension, then $M$ is projective in $\grmd$.
\end{lem}

\begin{proof}
Using Lemma \ref{lem:GRADEDLEMMA},
Proposition \ref{prop:presentation}(b) and
Proposition \ref{prop:ABELIANPROJECTIVE}(i), 
it is clear that we can use the proof of
\cite[Lemma 3.2.1]{lundstrom2004}
in the non-unital situation also.
\end{proof}

Now  we give a graded version of \cite[Proposition 2.2]{ARM}.

\begin{lem}\label{Re-grproyectivo}
With the above notations:
\begin{itemize}

\item[(a)] If $e \in \G_0$ and $u \in R_e$ is an idempotent,
then $R u$ is projective as an object in $\grmd$.

\item[(b)] If $R$ is object unital, then $R$ is projective 
as an object in $\grmod$.

\end{itemize}
\end{lem}

\begin{proof}
(a) Suppose that $g : M \to N$ and $h : Ru \to N$ 
are morphisms in $\grmd$ with $g$ surjective.
Since $u = u^2$, we get that $u \in Ru$.
Thus, from the surjectivity of $g$, it follows that 
there exists $m' \in M$ such that $g(m') = h(u)$.
Put $m = um'$. Then $g(m) = g(u m') = ug( m' ) = u h(u)
= h(u^2) = h(u)$.
Define $\overline{h} : Ru \to M$ by
$\overline{h}(ru) = r m$, for $r \in R$.
Now we show that $\overline{h}$ is well defined.
Suppose that $ru = 0$ for some $r \in R$.
Then $r m = r (u m') = (ru) m' = 0 m' = 0$.
It is clear that $\overline{h}$ is $R$-linear.
By Lemma \ref{lem:GRADEDLEMMA}, we can choose $\overline{h}$
so that it is graded.

(b) This follows from (a) above, Proposition \ref{prop:Ggroup}(a)
and Proposition \ref{prop:ABELIANPROJECTIVE}(a), since
$R = \oplus_{e \in \G_0} R(e) = \oplus_{e \in \G_0} R 1_e$
as objects in $\grmod$.
\end{proof}

\begin{rem} 
Notice that if $A$ is a non-unital ring then $_AA$ is not necessarily
projective. Indeed, $_AA$ is is in general  locally projective 
(see \cite[Proposition 2]{AnM}) .

\end{rem}

\begin{prop}\label{caracpro}
Suppose that $R$ is object unital and $P \in \grmod$. 
The following statements are equivalent:
\begin{itemize}

\item[(i)] $U(P)$ is projective.

\item[(ii)] $U(P)$ is projective in $R^1$-mod, 
where $R^1 = R \times \mathbb{Z}$ is the unitalization of $R.$

\item[(iii)] $P$ is projective.

\item[(iv)] Every short exact sequence in $\grmod$ 
\begin{equation}\label{procinde}
\begin{tikzcd}[column sep = small]
0 \arrow[r] & L \arrow[r, "f"] & M \arrow[r, "g"] & P \arrow[r] & 0
\end{tikzcd}
\end{equation}
splits.

\item[(v)] $P$ is a direct summand in $\grmod$ of a module which is free 
by suspension.
\end{itemize}

\end{prop}

\begin{proof}
(i)$\Leftrightarrow$(ii): 
This follows from (ii) of \cite[Proposition 2.4]{ARM} and the fact that $R$ is locally unital.

(i)$\Rightarrow$(iii): 
Consider the diagram $$\begin{tikzcd}[row sep = large, column sep = normal]
& P \arrow[d, "h"] \arrow[dl, dashed, "\overline{h}", swap] & \\ 
M \arrow[r, "g", swap] & N \arrow[r] &  0
\end{tikzcd}$$ 
of morphisms in $\grmod$ where $g$ is surjective. 
Since $P$ is projective in $\rmod$ there is 
$\overline{h}\colon P\to M$ in $\rmod$ such that $h=g\circ \overline{h}$. But then by  Lemma \ref{lem:GRADEDLEMMA} the map $\overline{h}$ can be considered in $\grmod$ and so $P$ is projective.

(iii)$\Rightarrow$(i): This can be done as the second part of the proof of \cite[Proposition 3.4.3.]{lundstrom2004}. 

(iii)$\Rightarrow$(iv): Consider the following diagram: $$\begin{tikzcd}[row sep = large, column sep = normal]
& & & P \arrow[d, "\id{P}"] \arrow[dl, dashed, "\varphi", swap] & \\ 
0\arrow[r] & L \arrow[r, "f", swap] & M \arrow[r, "g", swap] & P\arrow[r] & 0
\end{tikzcd}$$ If $P$ is projective, there is $\varphi\colon P\to M$ in $\grmod$ such that $g\circ\varphi=\id{P}$. By Proposition \ref{grcinde}, the sequence \eqref{procinde} splits.

(iv)$\Rightarrow$(v): By Proposition \ref{prop:presentation}(b) 
there is a short exact sequence 
$$\begin{tikzcd}[row sep = large, column sep = small]
0\arrow[r] & \operatorname{Ker}\varphi \arrow[r, hook] & F \arrow[r, "\varphi"] & P \arrow[r] & 0
\end{tikzcd}$$ with $F$ free by suspension. By hypothesis, this sequence split in $\grmod$, so $P$ is a direct summand of $F$.

(v)$\Rightarrow$(i): By Lemma \ref{lem:FREEIMPLPROJECTIVE} and Proposition \ref{prop:ABELIANPROJECTIVE} it follows that $P$ is projective. On the other hand, the proof of Lemma \ref{Re-grproyectivo} can be used to show that every free module by suspension is projective in $\rmod$ (specifically, every $R(\sigma)$, $\sigma\in\G$). Therefore, $U(P)$ being a direct summand of a direct sum of projective modules, is projective in $\rmod$.
\end{proof}

\begin{cor}\label{fg}
If $R$ is object unital and $P \in \grmod$, 
then $P$ is projective finitely generated if and only if $U(P)$ 
is projective finitely generated.
\end{cor}

\begin{proof}
This follows from Proposition \ref{prop:finitegeneration} 
and Proposition \ref{caracpro}.
%and  \cite[Proposition 3.5.1(b)]{lundstrom2004}.
\end{proof}

Recall that an object $A$ in an abelian category ${\mathcal A}$
is called injective if the functor
$\hom_{\mathcal{A}}(-,A) : {\mathcal A} \rightarrow \Ab$ is exact.
Similarly to the ungraded situation, there is a 
characterization of injective objects in $\grmd$: 

\begin{prop}
A module $Q \in \grmd$ is injective if and only if 
for every injective $f : M \to N$ in $\grmd$ and 
every $h : M \to Q$ in $\grmd$, 
there exists $\overline{h} : N \to Q$ in $\grmd$ 
such that $h= \overline{h} \circ f$
%, that is, making the diagram below commutative:
% $$\begin{tikzcd}[row sep = large, column sep = normal]
%& Q & \\ 
%0 \arrow[r] & M \arrow[u, "h"] \arrow[r, "f", swap] & N \arrow[ul, dashed, %"\overline{h}", swap]
%\end{tikzcd}$$
\end{prop}

\begin{prop}\label{prop:inj}
Let $(A_i)_{i \in I}$ be a family of objects in an abelian
category. Then $\prod_{i \in I} A_i$ is injective if and only if
each $A_i$ is injective.
\end{prop}

\begin{proof}
These are standard facts which can be found in \cite{S75}.
\end{proof}

Now we give a description of the injective objects in $\grmod$
analogous to Baer's criterion (see e.g. \cite{R79}).

\begin{prop}\label{inyecinde}
The following statements for $Q\in\grmod$ are equivalent:
\begin{itemize}
\item[i)] $Q$ is injective.
\item[ii)] Every short exact sequence in $\grmd$ 
\begin{equation}\label{inycinde}
\begin{tikzcd}[column sep = small]
0 \arrow[r] & Q \arrow[r, "f"] & M \arrow[r, "g"] & N \arrow[r] & 0
\end{tikzcd}
\end{equation}
splits.
\item[{\rm (iii)}] The functor 
${\rm HOM}_{\grmd}(-, M) : \rmd \rightarrow \Ab_{\G}$
is exact.
\item[{\rm (iv)}] For every graded left ideal $I$ of $R$,
the canonical map
$${\rm HOM}_R(R , M) \rightarrow {\rm HOM}_R(I , M)$$
is surjective.
\end{itemize}
\end{prop}

\begin{proof}
It is clear that if we use Proposition \ref{prop:inj}, 
then we can proceed as in the proof of 
\cite[Proposition 3.5.2]{lundstrom2004}
in the non-unital situation also.
\end{proof}

\begin{cor}\label{cor:UMMINJECTIVE}
If $M \in \grmd$ has $U(M)$ is injective, 
then $M$ is injective.
\end{cor}

\begin{rem}
The converse of Corollary \ref{cor:UMMINJECTIVE} does not
hold in general. For a counterexample when $\G$ is
a group, see \cite[p. 8]{N82}.
\end{rem}

Let $A$ be an object in an abelian category. Recall that a
subobject $B$ of $A$ is called essential (small) in $A$ if $B
\cap C \neq 0$ ($B + C \neq A$) for every nonzero subobject $C$
of $A$.

\begin{prop}\label{prop:NMESSENTIAL}
Suppose that $R$ is object unital and $M,N \in \grmod$
where $N$ is a graded submodule of $M$. 
\begin{itemize}
\item[{\rm (a)}] The module $N$ is essential in $M$
if and only if $U(N)$ is essential in $U(M)$.
\item[{\rm (b)}] If $U(N)$ is small in $U(M)$,
then $N$ is small in $M$.
\end{itemize}
\end{prop}

\begin{proof}
The proof of \cite[Proposition 3.6.1]{lundstrom2004} works
in the non-unital situation also.
\end{proof}

\begin{rem}
(a) The reversed implication in Proposition \ref{prop:NMESSENTIAL}(b) does
not hold in general. For a counterexample in the case when
$\G$ is a group, see \cite[p. 10]{N82}.

(b) It is not clear to the authors if Proposition \ref{prop:NMESSENTIAL}
holds for general groupoid graded rings $R$ and $M,N \in \grmd$.
\end{rem}

We say that $M \in \grmd$ is flat 
if the functor ${\rm -} \otimes_R M : \gmdr \to \Ab_\G$ is exact.

\begin{prop}
If $R$ is object unital and $M \in \grmod$, then
the following five statements are equivalent:
\begin{itemize}
\item[{\rm (i)}] The module $U(M)$ is flat.
\item[{\rm (ii)}] The module $M$ is flat.
\item[{\rm (iii)}] For every finitely presented 
$P \in \grmod$, the canonical graded map ${\rm HOM}_R(P,R) \otimes_R M
\rightarrow {\rm HOM}_R(P,M)$ is surjective.
\item[{\rm (iv)}] For every finitely presented $P \in \grmod$ 
and each semi-graded map $u: P \rightarrow M$,
there is $F \in \grmod$, free of finite type, such
that $U(F)$ is free of finite type, and semi-graded maps 
$v : P \to F$ and $w : F \rightarrow M$, 
such that $u = w \circ v$.
\item[{\rm (v)}] The module $M$ is the direct limit of 
$F_i \in \grmod$, $i \in I$, of finite type, such that each
$U(F_i)$ is free of finite type.
\end{itemize}
\end{prop}

\begin{proof}
The proof of \cite[Proposition 3.7.2]{lundstrom2004}
works in the non-unital situation also.
\end{proof}

\section{Semisimplicity}\label{sec:semisimplicity}

Throughout this section, $R$ denotes an object unital $\G$-gradd ring.
In classical module theory, given $M\in\rmod$, 
the following properties are equivalent:
\begin{itemize}
\item $M$ is semisimple;
\item $M$ is a direct sum of simple modules; 
\item every submodule of $M$ is a direct summand.
\end{itemize}
In this section, we prove a graded version of this result
(see Proposition \ref{ss2}). 
After that, we show that semisimplicity is a well behaved
under the preimage of the forgetful functor 
(see Proposition \ref{preimageU}).
At the end of this section, we obtain a result
relating graded simplicity of $R$ (as an object in $\grmod$)
to graded injectivity and graded projectivity 
(see Proposition \ref{equivalent}).

\begin{defi}\label{def:semisimple}
Let $M \in \grmod$.
We say that $M$ is {\it simple} if $\{ 0 \}$ and $M$
are the only graded submodules of $M$.
We say that $M$ is {\it semisimple} if $M$ is the direct sum
of a family of simple modules in $\grmod$.
\end{defi}

\begin{lem}\label{as}
Let $M \in \grmod$ and suppose that 
$M = \sum_{i\in I} M_i$, where, for all $i \in I$, 
$M_i \in \grmod$ is a simple module.
If $N \in \grmod$ is a graded submodule of of $M$, 
then there exists $J\subseteq I$ such that 
$M=N\oplus\bigoplus_{j\in J}M_j.$
\end{lem} 

\begin{proof}
Consider the non-empty set 
$$\mathcal{M}=\left\{J\subseteq I: N+\sum_{j\in J} 
M_j\textup{ is direct}\right\}.$$ 
The set $\mathcal{M}$ is partially ordered by inclusion. 
We claim that $\mathcal{M}$ is inductive.
If we assume that the claim holds, then by an application of 
Zorn's Lemma, the maximal element of $\mathcal{M}$ is exactly $M$.

Now we show the claim.
Let $\mathcal{B}$ be a chain in $\mathcal{M}$. Put $K=\bigcup_{J\in\mathcal{B}}J$. Then $K$ is an upper bound for $\mathcal{B}$ and $K\in\mathcal{M}$. For, if $K\not\in\mathcal{M}$, there will be $j_1\in J_1,\ldots,j_r\in J_r$ and $n\in N,m_1\in M_{j_1},\ldots, m_{j_r}\in M_{j_r}$ not all zero such that $$0=n+\sum_{t=1}^{r} m_{j_t}.$$ Being $\mathcal{B}$ chain, there is $J\in\mathcal{B}$ such that $j_1,\ldots,j_r\in J$ and consequently the sum $N+\sum_{j\in J} M_j$ will not be direct, which is impossible. Therefore,  Zorn's Lemma provides a maximal $J\subseteq I$ with the property that $L=N+\sum_{j\in J} M_j$ is direct. If we can show that $L=M$ we will be end. For this, is enough to see that $M_i\subseteq L$, for every $i\in I$. But this follows immediately since if $M_i\cap L=\{0\}$ then $J\cup \{i\}\in\mathcal{M}$.
\end{proof}

\begin{prop}\label{ss1}
The following properties of an object $M\in\grmod$ are equivalent:
\begin{itemize}
\item[(i)] $M$ is semisimple.
\item[(ii)] $M$ is a direct sum of simple modules.
\end{itemize}
\end{prop}

\begin{proof}
The implication (ii)$\Rightarrow$(i) follows from
Definition \ref{def:semisimple}.
%From the definition is straightforward that ii)$\Rightarrow$i). 
Suppose that (i) holds.
Let $M=\sum_{i\in I}M_i$ be a sum of simples modules. 
The claim now follows from Lemma \ref{as} by taking $N=\{0\}$.
\end{proof}

%The next result shows two really important facts about  graded submodules of a semisimple graded module.

\begin{prop}\label{directt}
Let $M=\bigoplus\limits_{i\in I} M_i$ be a sum of simple modules and let $N$ be a graded submodule of $M$. Then:
\begin{itemize}
\item[(i)] $N$ is a direct summand.
\item[(ii)] $N\cong \bigoplus_{j\in J} M_j$, for some $J\subseteq I$ and the isomorphism is given in $\grmod$.
\end{itemize}
\end{prop}

\begin{proof}
(i) follows directly from the proof of Lemma \ref{as}. 
Now we prove (ii). 
By (i), there is a graded submodule $K$ of $M$ such that $M=N\oplus L$. By an application of Zorn's Lemma as in the proof of Lemma \ref{as}, $M=L\oplus\bigoplus_{j\in J} M_j$ for some $J\subseteq I$. Therefore, $N\cong M/L\cong\bigoplus_{j\in J}M_j$ and the isomorphism is given in $\grmod$.
\end{proof}

%As a consequence of Proposition \ref{directt} we have  the following %fact.

\begin{cor}
Every graded submodule and every quotient of a semisimple module is semisimple.
\end{cor}

\begin{proof}
This follows immediately from Proposition \ref{directt}.
\end{proof}

\begin{lem}\label{suspension-ideal}
If $M$ is an object in $\grmod$ and $M$ is finitely generated, as an object in $\grmod$, then $M$ contains a maximal graded submodule.
\end{lem}

\begin{proof}
Consider the collection $\mathcal{L}$ of all proper graded submodules of $M$. This is a non-empty set partially ordered by inclusion. Let $\mathcal{K}$ be a chain in $\mathcal{L}$ and put $N=\sum_{L\in\mathcal{K}}L$. Then $N$ is an upper bound for $\mathcal{K}$ and $N\in\mathcal{L}$. Otherwise $N=M$ and there would be a finite set $I\subseteq H(M)$ such that $N=\sum_{m\in I}Rm$. But then every $m\in I$ would belong to the graded submodule $Rm$ in $\mathcal{K}$ and since $\mathcal{K}$ is a chain, the finite sum $N=\sum_{m\in I}Rm\in\mathcal{K}$, leading to the contradiction $M\in\mathcal{K}$. From this, Zorn's Lemma provides a maximal graded submodule of $M$.
\end{proof}

\begin{prop}\label{ss2}
For an object $M\in\grmod$ the following properties are equivalent:
\begin{itemize}
\item[(i)] $M$ is semisimple
\item[(ii)] $M$ is a direct sum of simple modules.
\item[(iii)] Every graded submodule of $M$ is a direct summand.
\end{itemize}
\end{prop}

\begin{proof} 
The equivalence (i)$\Leftrightarrow$(ii) follows
from Proposition \ref{directt}. 
%we see that i) and ii) are equivalent and 
The implication (i)$\Rightarrow$(iii) follows 
from Lemma \ref{as}.
Finally, we prove the implication 
(iii)$\Rightarrow$(i). 
Suppose that every graded submodule of $M$ is a direct summand.
Then, in particular, $L=\sum\{N\colon N\textup{ is a graded simple submodule of } M\}$ is a direct summand of $M$. 
It is enough to show that the complement of $L$ is $\{0\}$. 
To this end, note that every graded submodule of $M$ contains a simple submodule. In fact, since every graded submodule is a sum of homogeneous cyclic modules, is enough to see this assertion is valid for every $Rm$, $m\in H(M)$. Given $m\in H(M)$ since $R$ is object unital then $Rm$ is finitely generated and 
then by Lemma \ref{suspension-ideal} there exists 
a maximal graded submodule $K$ of $Rm$. By hypothesis, $M=K\oplus K^{\prime}$ with $K^{\prime}$ graded submodule of $M$. But $Rm=M\cap Rm=K\oplus (K^{\prime}\cap Rm)$, then $Rm\cap K^{\prime}\cong Rm/K$ is a simple submodule of $Rm$ due to the maximality of $K$ over $Rm$. Summarizing, $L=M$ is a sum of simple modules. 
\end{proof}

\begin{prop}\label{preimageU}
Let $M\in\grmod$. If $U(M)$ is semisimple then $M$ is semisimple.
\end{prop}

\begin{proof}
Let $N$ be a graded submodule of $M$. If $U(M)$ is semisimple then $U(N)$ is a  direct summand, so there is $f\colon M\to N$ in $\rmod$ such that $f\circ \iota_N=\id{N}$, where $\iota_N\colon N\to M$ is the canonical inclusion. 
By Lemma \ref{lem:GRADEDLEMMA} we can assume that
$f$ is graded and hence
% can be chosen considered in $\grmod$ by Lemma \ref{lem:GRADEDLEMMA} 
$N$ is a direct summand of $M$ in $\grmod$.
\end{proof}

\begin{defi}
We say that $R$ is {\it semisimple} as a graded ring
if it is semi\-simple considered as an object in $\grmod$.
\end{defi}

\begin{prop}\label{equivalent}
If $R$ is an object-unital ring, 
then the following properties are equivalent:
\begin{itemize}
\item[(i)] The ring $R$ is semisimple as a graded ring.
\item[(ii)] Every graded left ideal $I$ of $R$ is a direct summand
of $R$.
\item[(iii)] Every object in $\grmod$ is injective.
\item[(iv)] Every object in $\grmod$ is projective.
\item[(v)] Every object in $\grmod$ is semisimple.
\end{itemize}
\end{prop}

\begin{proof}
(i)$\Rightarrow$(ii): 
This follows from Proposition \ref{ss2}.

(ii)$\Rightarrow$(iii): 
Let $M\in\grmod$, $I,J$ be graded left ideals of $R$ such that $R=I\oplus J$ and $g\in {\rm HOM}_R(I,M)$. The function $f\colon R\to M$ defined by $f(i+j)=g(i),$ for every $i\in I$ and every $j\in J$, satisfies $g=f\circ \iota$, where $\iota\colon I\to R$ is the inclusion, and $f\in {\rm HOM}_R(R,M)$. By Baer's Criteria (Proposition \ref{inyecinde}), $M$ is injective.

(iii)$\Rightarrow$(iv): 
If \[ \begin{tikzcd}[cramped, column sep = small]
0 \arrow[r] & L \arrow[r] & M \arrow[r] & N \arrow[r] & 0
\end{tikzcd} \] is a short exact sequence in $\grmod$, that $L$ is injective implies by Proposition \ref{inyecinde} that the sequence splits. But this is equivalent to (iv) by Proposition \ref{caracpro}.

(iv)$\Rightarrow$(v): 
For every object in $\grmod$, any graded submodule induces a short exact sequence that split by hypothesis, turning it into a direct summand.

(v)$\Rightarrow$(i): This is clear. 
\end{proof}


\begin{thebibliography}{R}

\bibitem{AnM}P. N. \'Anh and L. M\'arki, \emph{Morita equivalence for rings without identity,}  Tsukuba J.Math. 11(2) (1987), 1–16.

\bibitem{ARM} G. Aranda, K M. Rangaswamy  and 
M. S. Molina, \emph{Weakly Regular and Self-Injective Leavitt Path
Algebras Over Arbitrary Graphs}, Algebr. Represent. Theor, (2011) 14:751–777.
%
%\bibitem{bagio2010}
%D. Bagio, D. Flores and A. Paques, 
%Partial actions of ordered groupoids on rings. 
%{\it J. Algebra Appl.} {\bf 9}, 501-–517 (2010).

\bibitem{bagio2012}
D. Bagio and A. Paques, Partial Groupoid Actions: Globalization, Morita Theory, and
Galois Theory. {\it Comm. Alg.} {\bf 40}, 3658-–3678 (2012).

\bibitem{batista2017}
E. Batista, Partial actions: what they are and why we care,
{\it Bull. Belg. Math. Soc. Simon Stevin}
Volume 24, Number 1 (2017), 35--71.

\bibitem{DB}  D. A. Buchsbaum, Exact Categories and Duality, \textit{Transactions of the American Mathematical Society} \textbf{1}, 1--34 (1955).

\bibitem{CNP}
J. Cala, P. Lundstr\"{o}m and H. Pinedo,
Object-unital groupoid graded rings, crossed products and separability.
Accepted for publication in {\it Communications in Algebra}.

\bibitem{Fu} K. R. Fuller, \emph{On rings whose left modules are direct sums of finitely generated modules}, Proc. Amer. Math. Soc. 54 (1976), 39–44.

%\bibitem{RH}  R. Hazrat, \emph{Graded Rings and Graded Grothendieck Groups} (London Mathematical Society Lecture Note Series: 435, 2016)

\bibitem{L} M. V. Lawson, \emph{Inverse Semigroups. The Theory of Partial Symmetries} (World Scientific Pub. Co., 1995).

%\bibitem{liu2005}
%G. Liu, F. Li,
%On Strongly Groupoid Graded Rings and the
%Corresponding Clifford Theorem,
%{\it   Alg. Colloquium} \textbf{13}(2), 181--196 (2005).

\bibitem{lundstrom2004}
P. Lundstr\"{o}m,
The category of groupoid graded modules,
{\it Colloq. Math.} \textbf{100}(4), 195--211 (2004).

\bibitem{lundstrom2005}
P. Lundstr\"{o}m,
Crossed product algebras defined by separable extensions,
{\it J. Algebra} {\bf 283} (2005), 723--737.

\bibitem{lundstrom2006}
P. Lundstr\"{o}m,
Strongly groupoid graded rings and cohomology,
{\it Colloq. Math.} \textbf{106}(1), 1--13 (2006).

\bibitem{N82}
C. N{\v a}st{\v a}sescu and F. Van Oystaeyen, {\it Graded Ring
Theory}, North-Holland Publishing Co., Amsterdam-New York (1982).

\bibitem{nastasescu2004}
C. Nastasescu and F. van Oystaeyen,
{\it Methods of graded rings},
Springer Lecture Notes (2004).

\bibitem{Ny2019} P. Nystedt, \emph{A survey of $s$-unital and locally unital rings.}
{\it Revista Integraci\'on.} {\bf 37} (2), 251--260 (2019).  

\bibitem{NyOP22018}  P. Nystedt, J. \"{O}inert, H. Pinedo, Epsilon-strongly groupoid graded rings, the Picard inverse category and cohomology,   {\it Glasgow Math. J.} {\bf 62} (1), 233--259 (2020).  

\bibitem{popescu1973}
N. Popescu.
{\it Abelian Categories with Applications to Rings and Modules.} 
Academic Press, Inc. (1973)

\bibitem{ren} J. Renault,
              {\it A Groupoid Approach to $C^*$-algebras},
              Lecture Notes in Mathematics {\bf 793}(2) (1980).

\bibitem{R79}
J. Rotman, {\it An Introduction to Homological Algebra}, Academic
Press (1979).

\bibitem{S75}
B. Stenstr\"{o}m, {\it Rings of Quotients}, Springer-Verlag, New
York-Heidelberg (1975).

\bibitem{tominaga1976}
H. Tominaga H.,
On $s$-unital rings, 
{\it Math. J. Okayama univ.} {\bf 18} (1976), 117-–134.

\end{thebibliography}
\end{document}